\documentclass[leqno,draft]{article}

\usepackage{amsmath,amsthm,amssymb}
\usepackage[mathscr]{eucal}
\allowdisplaybreaks[4]

\newcommand{\Cyclic}{\mathop{\text{\Large$\mathfrak S$}\vrule width 0pt%
depth 1pt}}

\makeatletter
\newcommand{\inp}{\@ifstar{\inpFixed}{\inpScaled}}
\newcommand{\inpScaled}[3][]{\def\@tmp{#1}\def\@tmpd{}%
\ifx\@tmp\@tmpd\def\@tmpl{left}\def\@tmpr{right}%
\else\def\@tmpl{#1l}\def\@tmpr{#1r}\fi \csname \@tmpl \endcsname
\langle #2, #3 \csname \@tmpr \endcsname \rangle}
\newcommand{\norm}{\@ifstar{\normFixed}{\normScaled}}
\newcommand{\normScaled}[2][]{\def\@tmp{#1}\def\@tmpd{}%
\ifx\@tmp\@tmpd\def\@tmpl{left}\def\@tmpr{right}%
\else\def\@tmpl{#1l}\def\@tmpr{#1r}\fi \csname \@tmpl \endcsname
\lVert #2 \csname \@tmpr \endcsname \rVert}
\makeatother

\theoremstyle{plain}
\newtheorem{theorem}{Theorem}[section]
\newtheorem{lemma}[theorem]{Lemma}
\newtheorem{corollary}[theorem]{Corollary}
\newtheorem{proposition}[theorem]{Proposition}
\theoremstyle{definition}

\theoremstyle{remark}
\newtheorem{remark}[theorem]{Remark}

\title{\Large \bf The canonical $8$-form on manifolds\\ with holonomy group $\mathrm{Spin}(9)$
}

\author{\large\sc M.\,Castrill\'on L\'opez, P.\,M. Gadea, \& I.\,V. Mykytyuk
\thanks{Partially supported by the Ministry of Science and
Innovation, Spain, under Project MTM2008--01386.}}

\date{}

\begin{document}

\maketitle

\begin{abstract}
An explicit expression of the canonical
$8$-form on a Riemannian manifold with a
$\mathrm{Spin}(9)$-structure, in terms of the nine local symmetric
involutions involved, is given. The list of explicit
expressions of all the canonical
forms related to Berger's list of holonomy groups is
thus completed. Moreover, some results on
$\mathrm{Spin}(9)$-structures as
$G$-structures defined by a tensor and on the curvature tensor of
the Cayley planes, are obtained.
\end{abstract}

\section{Introduction and Preliminaries}
\label{se.1}
\setcounter{equation}{0}

The group
$\mathrm{Spin}(9)$ belongs to Berger's list \cite{Ber1} of
restricted holonomy groups of locally irreducible Riemannian
manifolds which are not locally
symmetric. Manifolds with holonomy group
$\mathrm{Spin}(9)$ have been studied by Alekseevsky \cite{Ale},
Brown and Gray \cite{BroGra}, Friedrich \cite{Fri1,Fri2}, and Lam
\cite{Lam}, among other authors. As proved in \cite{BroGra,Ale},
a connected, simply-connected, complete non-flat
$\mathrm{Spin}(9)$-manifold is isometric
to either the Cayley projective plane
$\operatorname{\mathbb O \mathrm P}(2) \cong F_4/\mathrm{Spin}(9)$
or its dual symmetric space, the Cayley hyperbolic plane
$\operatorname{\mathbb O \mathrm H}(2) \cong F_{4(-20)}/\mathrm{Spin}(9)$.

Moreover, $\Delta_9$ being the unique irreducible $16$-dimensional
$\mathrm{Spin}(9)$-module, the $\mathrm{Spin}(9)$-module
$\Lambda^8(\Delta^*_9)$ contains one
and only one (up to a non-zero factor)
$8$-form $\Omega^8_0$ which is
$\mathrm{Spin}(9)$-invariant and
defines the unique parallel form on
$\operatorname{\mathbb O \mathrm P}(2)$.
It induces a {\it canonical $8$-form\/}
$\Omega^8$ on any
$16$-dimensional manifold with a fixed
$\mathrm{Spin}(9)$-structure. This form is said to be canonical
because (cf.\ \cite[p.\ 48]{BroGra}, Berger \cite[p.\ 13]{Ber2})
it yields, for the compact case, a generator of
$H^8(\operatorname{\mathbb O \mathrm P}(2),\mathbb{R})$.

Some explicit expressions of
$\Omega^8$ have been given. The first one by Brown and Gray in
\cite[p.\ 49]{BroGra} in terms of a Haar integral. Other
expression was then given by Brada and P\'ecaut-Tison~\cite[pp.\
150, 153]{BraPec}, by using a ``cross product.'' Unfortunately,
their formula is not correct, as we explain in Appendix A. Another
expression was then given by Abe and
Matsubara in~\cite[p.\ 8]{AbeMat} as a sum of
$702$ suitable terms (see also Abe \cite{Abe}). Their formula
contains some errors, see Appendix B below.

In this paper we give (Theorem \ref{th.1.1})
an explicit expression of the canonical
$8$-form $\Omega^8$ on a
$\mathrm{Spin}(9)$-manifold, in terms of the nine local symmetric
involutions involved.

On the one hand, this completes the list of canonical forms which
are related to Berger's list of holonomy
groups (for the Kraines form \cite{Kra} for
$\mathrm{Sp}(n)\mathrm{Sp}(1)$ and the Bonan forms \cite{Bon} for
$G_2$ and
$\mathrm{Spin}(7)$ see also, e.g.\ Salamon \cite[pp.\ 126, 155,
173]{Sal}). On the other hand, we furnish
an explicit analogue to the K\"ahler
$2$-form
$\Omega^2$ and quaternion-K\"ahler
$4$-form
$\Omega^4$, which can in a sense be called their octonionic
analogue, as follows.

We recall that a $\mathrm{Spin}(9)$-structure on an connected, oriented
$16$-dimensional Riemannian manifold
$(M,g)$ is defined as a reduction of its bundle of oriented
orthonormal frames $\mathrm{SO}(M)$, via the spin
representation $\rho (\mathrm{Spin}(9))\subset \mathrm{SO}(16)$.
Equivalently (Friedrich \cite{Fri1,Fri2}), a
$\mathrm{Spin}(9)$-structure is given by nine-dimensional subbundle
$\nu^9$ of the bundle of endomorphisms $\mathrm{End}(TM)$ locally
spanned by $I_i \in \Gamma(\nu^9)$, $0 \leqslant i \leqslant 8$,
satisfying the relations $I_i I_j + I_j I_i = 0$, $i \neq j$,
$I^2_i = \mathrm{I}$, $I^T_i = I_i$, $\mathrm{tr}\; I_i = 0$,
$i,j=0,\dots,8$. These endomorphisms define $2$-forms
$\omega_{ij}$, $0\leqslant i < j\leqslant 8$, on $M$ locally by
$\omega_{ij}(X,Y) = g(X,I_i I_j Y)$. Similarly, using the
skew-symmetric involutions $I_i I_j I_k$, $0\leqslant
i<j<k\leqslant 8$, one can define $2$-forms $\sigma_{ijk}$. The
$2$-forms $\{ \omega_{ij}, \sigma_{ijk} \}$ are linearly
independent and a local basis of the bundle $\Lambda^2M$.

The main purpose of the present paper is to prove
\begin{theorem}\label{th.1.1}
The canonical $8$-form on the
$\mathrm{Spin}(9)$-manifold
$(M,g,\nu^9)$ is given by \[ \Omega^8 = \sum_{\substack{
i,j=0,\dots,8\\ i',j'=0,\dots,8}} \omega_{ij} \wedge
\omega_{ij'} \wedge \omega_{i'j} \wedge \omega_{i'j'},\]
where $\omega_{ij}=-\omega_{ji}$ if $i>j$ and $\omega_{ij}=0$
if $i=j$.
\end{theorem}
On the other hand, some expressions for the curvature tensors of
the Cayley planes have been given (cf.\ Brown and
Gray~\cite{BroGra}, Brada and P\'ecaut-Tison \cite{BraPec0,BraPec},
and~\cite{Myk1,Myk2}). As an application of our Theorem
\ref{th.1.1} we give one expression in terms of the nine local
symmetric operators and relate it to the other expressions.

The importance of the Cayley planes in
geometry is well known. Moreover, both the group
$\mathrm{Spin}(9)$ and the
$\mathrm{Spin}(9)$-structures do appear in some questions of
Physics, and we now recall some of them. The space
$\operatorname{\mathbb O \mathrm H}(2)$ is the only solution to
$N=9$, $d=16$,
$3$-dimensional supergravity (cf.\ de Wit,
Tollst\'en, and Nicolai \cite{WitTolNic}). The group
$\mathrm{Spin}(9)$ appears in M-theory (see
Banks et al.\ \cite{BanFisSheSus}), related to
$16$ fermionic superpartners, transforming as spinors under
$\mathrm{SO}(9)$, linked to the very short strings connecting a
system of D0 branes. Furthermore, Sati
\cite{Sat1,Sat2} has recently studied the relation of
$\mathrm{Spin}(9)$-structures with M-theory fields, proving that
the massless fields of M-theory
are encoded in the spinor bundle of
$\operatorname{\mathbb O \mathrm P}(2)$
and that the massless multiplet of
$11$-dimensional supergravity is related to
$\operatorname{\mathbb O \mathrm P}(2)$ bundles over
eleven-manifolds. In addition, the canonical 8-form
$\Omega ^8$ is there used to define a term of the action
functional given in the theory. We remark
that, besides the theoretical expression of
$\Omega ^8$ given in \cite{BroGra}, the flawed expressions in
\cite{BraPec,AbeMat} are mentioned in \cite{Sat2}.

As for the contents of this paper, in \S
$2$, after recalling some properties of
$\mathrm{Spin}(9)$-manifolds and the nine local symmetric
involutions involved, we obtain the aforementioned expression for
$\Omega^8$ and then some corollaries. In \S
$3$ we apply the previous results to the definition of a
$\mathrm{Spin}(9)$-structure as a structure
defined by a tensor. We deduce in \S
$4$ some results on the curvature tensor of the Cayley planes.
Finally, the aforementioned appendices A and B follow.

\section{The canonical
$8$-form in terms of the nine local symmetric involutions}
\setcounter{equation}{0}

In order to prove Theorem \ref{th.1.1},
we first study the action of the group
$\mathrm{Spin}(9)$ on
$\mathbb{R}^{16} \equiv \mathbb{O}^2$ in
terms of the nine local symmetric involutions
$I_i$.

\subsection{The action of
$\mathrm{Spin}(9)$ on
$\mathbb{R}^{16} \equiv \mathbb{O}^2$}

The isotropy representation of either
$\operatorname{\mathbb O \mathrm P}(2)$ or
$\operatorname{\mathbb O \mathrm H}(2)$
is known to be isomorphic to the
$16$-dimensional spin representation
$\rho$ of
$\mathrm{Spin}(9)$.

Let
$V^9$ be a real vector space of dimension nine
endowed with a positive definite bilinear form
$Q$. Let
$e_0,\dotsc,e_8$ be an orthonormal basis of
$V^9$. The Clifford algebra
$\mathrm{Cl}_+(9)$ in terms of this basis is
defined as the real associative algebra with unit
$1$, generators
$e_0,\dotsc,e_8$, and defining relations
\[ e_i \cdot e_j + e_j
\cdot e_i = 0, \quad i\neq j, \qquad e^2_i
= 1, \quad i,j=0,\dots,8.
\]
Let
$\mathrm{Pin}_+(9)$ be the multiplicative subgroup
of the group of all the invertible elements of
$\mathrm{Cl}_+(9)$ generated by the vectors of length one in
$V^9$. If $Q(v,v) = 1$ then
$v \cdot v=1$, so
$v\in \mathrm{Pin}_+(9)$. The Lie group
$\mathrm{Spin}_+(9)$, which we denote simply by
$\mathrm{Spin}(9)$, as they are isomorphic (cf.\ Postnikov
\cite[Lect.\ 13, Rem.\ 2]{Pos}), is the subgroup of
$\mathrm{Pin}_+(9)$ consisting of even elements, i.e. \[
\mathrm{Spin}(9) = \{\, v_1 \cdot v_2 \cdot \ldots \cdot v_{2k},\;
Q(v_i,v_i)=1, \; i=1,\dotsc,2k, \; k \in
\mathbb{N} \,\}. \] Moreover, the group
$\mathrm{Spin}(9)$ preserves under conjugation the space
$V^9$, that is, $sV^9s^{-1}=V^9$ for all
$s \in \mathrm{Spin}(9)$ (cf.\
\cite[Lect.\ 13]{Pos}). We denote by
$\pi$ the corresponding representation of the group
$\mathrm{Spin}(9)$ on $V^9$. Then
$\pi(\mathrm{Spin}(9))= \mathrm{SO}(9)$ and
$\pi\colon \mathrm{Spin}(9) \to \mathrm{SO}(9)$ is the usual
two-fold covering homomorphism (cf.\ \cite[Lect.\ 13]{Pos}).

There exists a faithful representation $\rho$ of
$\mathrm{Pin}_+(9)$ by orthogonal matrices (cf.\ \cite[Lect.\
13]{Pos}). In other words, $\rho(\mathrm{Pin}_+(9))\subset
\mathrm{O}(16)$ and $\rho(\mathrm{Spin}(9))\subset
\mathrm{SO}(16)$. Therefore, there exist nine orthogonal linear
transformations $I_i$ of $\Delta_9 = \mathbb{R}^{16}$ satisfying
the relations
\begin{equation}\label{eq.2.1}
I_i I_j + I_j I_i = 0, \; i \neq j,\quad I^2_i = \mathrm{I}, \quad
I^T_i = I_i, \quad \mathrm{tr}\; I_i = 0, \qquad i,j=0,\dots,8.
\end{equation}

The set
$\{I_iI_j, 0\leqslant i<j\leqslant 8\}$
is a basis of the Lie algebra
$\rho_*(\mathfrak{spin}(9))\subset\mathfrak{so}(16)$. Indeed,
since
$$ [I_iI_j, I_k]  =\left\{
\begin{array}{ll}
0, & \text{if}\ k\not= i,j, \\
-2I_j, & \text{if}\ k=i, \\
2I_i, & \text{if}\ k=j,
\end{array}
\right.
$$
the operators
$I_iI_j$ are linearly independent and
generate a space of dimension equal to
$\dim\mathfrak{so}(9)$. Taking into account that each operator
$I_iI_j$ is the tangent vector at
$t=0$ to the curve
$$
s(t) = \bigl( \cos (t/2) I_i-\sin (t/2) I_j\bigr)\bigl(\cos (t/2)
I_i+\sin (t/2) I_j\bigr) =\cos t \cdot \mathrm{I} +\sin t \cdot
I_iI_j
$$
in
$\rho(\mathrm{Spin}(9))$ passing through
the identity I, we obtain that the operators
$I_iI_j$ generate the Lie algebra
$\rho_*(\mathfrak{spin}(9))$ and, consequently,
by the connectedness of the Lie group
$\mathrm{Spin}(9)$ the following proposition holds

\begin{proposition}\label{pr.2.1}
The Lie group
$\rho(\mathrm{Spin}(9))\subset \mathrm{SO}(16)$ is generated by
the one-parameter families of endomorphisms \[ \exp(t I_iI_j)=\cos
t \cdot \mathrm{I} + \sin t\cdot I_iI_j, \qquad 0\leqslant
i<j\leqslant 8, \quad t\in\mathbb{R}. \]
\end{proposition}

In the sequel, we shall denote
$I_iI_j$ simply by
$I_{ij}$ and so on.

Let $(M,g,\nu^9)$ be a $\mathrm{Spin}(9)$-manifold, $p\in M$ and
$I_i,0\leqslant i \leqslant 8,$ a local basis of sections of $\nu
^9$ around $p$ satisfying the relations \eqref{eq.2.1}. Then, there
exists an isomorphism between $\mathbb{O}^2 \equiv \mathbb{R}^{16}$
and $T_p M$ such that the restriction of $g$ at $p\in M$ induces
the standard scalar product $\inp{\cdot}{\cdot}$ of $\mathbb{O}^2$,
given by
\begin{equation}\label{eq.2.2}
\inp{(x_1,x_2)}{(y_1,y_2)} = \inp{x_1}{y_1} + \inp{x_2}{y_2},
\quad \inp{x_a}{y_a}=\frac12(x_a\bar y_a+y_a\bar x_a),
\end{equation}
for $a=1,2$, and the endomorphisms $I_0,\dotsc,I_8$ of
$\mathbb{O}^2\equiv T_p M$ read
\begin{equation}\label{eq.2.3}
I_i(x_1,x_2) = (u_i\bar{x}_2,\bar{x}_1u_i), \quad I_8(x_1,x_2) =
(-x_1,x_2), \qquad (x_1,x_2)\in\mathbb{O}^2,
\end{equation}
where $u_0=1\in\mathbb{O}$ and $u_i$, $i=1,\dotsc,7$, stand for the
imaginary units of $\mathbb{O}$. One can easily check that these
endomorphisms satisfy the appropriate relations~\eqref{eq.2.1} (see
Postnikov \cite[Lect.\ 15]{Pos} and~\cite[(3),(4)]{Myk2}).

Moreover, as seen in Proposition \ref{pr.2.1}, the group
$\rho(\mathrm{Spin}(9))$ acting on $\mathbb{R}^{16} \equiv
\mathbb{O}^2$ is generated by the endomorphisms $M_{kl}^t=\cos t
\cdot \mathrm{I} +\sin t\cdot I_{kl}$, for $0\leqslant k<l\leqslant
8$, and it is a subgroup of the group $\mathrm{SO}(16)$ determined
by the standard scalar product \eqref{eq.2.2} of $\mathbb{O}^2$.

\subsection{Proof of Theorem \ref{th.1.1}}

We must prove that the $8$-form
$\Omega^8_0=\Omega^8|T_pM$, for an arbitrarily fixed point
$p\in M$, is
$\mathrm{Spin}(9)$-invariant and non-trivial.

{\it The $8$-form $\Omega^8_0$ is $\mathrm{Spin}(9)$-invariant.}
Fix a pair $kl$, $0\leqslant k<l\leqslant8$ and consider the action
of the endomorphism $M^t_{kl}$ on the set of forms
$\{\varpi_{ij}=\omega_{ij}|{T_pM},i,j=0,\dots,8\}$.
Remark that $\varpi_{ij}=0$ if $i=j$. Denote by $\overline{D}$ the
set of all the ordered pairs $ij$, where $i,j=0,\dots,8$
and $i\not=j$. We will call a subset $r_i=\{i'j'\in\overline{D}
: i'=i\}$ of the set $\overline{D}$ (resp.\ $c_j=\{i'j'
\in\overline{D} : j'=j\})$ an $i$-row (resp.\ a $j$-column). We
also consider the short $k$-row $r^*_k=r_k\setminus\{kl\}$ and the
short $k$-column $c^*_k=c_k\setminus\{lk\}$ (this time
$\#(c^*_k)=\#(r^*_k)=7$). Similarly one determines the short
$l$-row and the short $l$-column. Put
\begin{align*}
A_0 &= \{ij\in \overline{D} : \{k,l\}\cap\{i,j\}=\emptyset\}, \\
A_2 &= \bigl\{ ij\in \overline{D} :
\{k,l\}\cap\{i,j\}=\{k,l\}\bigr\} = \{kl, lk\}, \\ A_1^+ &=
\{ij\in \overline{D} : \{k,l\}\cap\{i,j\}=\{k\}\} = r^*_k\sqcup
c^*_k,\\ A_1^- &= \bigl\{ ij\in \overline{D} :
\{k,l\}\cap\{i,j\}=\{l\} \bigr\} = r^*_l\sqcup c^*_l,\\ P_{kl}&=
r_k\cup c_k\cup r_l\cup c_l = A_1^+\sqcup A_1^-\sqcup A_2,
\end{align*}
where we denote the union of two sets $A$ and $B$ by $A\sqcup B$ if
$A\cap B=\emptyset$. It is clear that $\overline{D}=A_0\sqcup
A_2\sqcup A_1^+\sqcup A_1^-$. Given a pair $ij\in \overline{D}$, we
denote by $\widehat{ij}$ the new pair obtained by replacing the
element $k$ (if it occurs in $ij$) by $l$ and the element $l$ (if
it occurs in $ij$) by $k$. The correspondence $ij\mapsto
\widehat{ij}$ defines a bijection $\mu\colon
\overline{D}\to\overline{D}$. It is clear that
$\mu(A_1^\pm)=A_1^\mp$ and this mapping is an involutive
automorphism of the set $\overline{D}$. In particular,
$\widehat{ij}=ij$ for $ij\in A_0$ and $\widehat{kl}=lk$.

By definition, for arbitrary
$X,Y\in\mathbb{O}^2$, we have \[
\bigl((M^t_{kl})^*\varpi_{ij}\bigr)(X,Y) = \inp{(\cos t+\sin t
\cdot I_{kl}) X}{(\cos t+\epsilon\sin
t\cdot I_{kl})I_{ij}Y}, \] where
$\epsilon=1$ if the number of common elements in the sets
$\{i,j\}$ and $\{k,l\}$ is even, and
$\epsilon=-1$ if it is odd. Taking
into account that all the operators
$I_i$ are orthogonal and that the operator
$I_{kl}$ is skew-symmetric, it is easily seen that
\begin{equation}\label{eq.2.4}
(M^t_{kl})^*\varpi_{ij}= \left\{
\begin{array}{ll}
\varpi_{ij}, &\text{if}\quad ij \in A_0\cup A_2,\\ \cos 2t\cdot
\varpi_{ij}+\sin 2t\cdot \varpi_{\widehat{ij}}, & \text{if}\quad
ij\in A_1^+,\\ \cos 2t\cdot  \varpi_{ij}-\sin 2t\cdot
\varpi_{\widehat{ij}}, & \text{if}\quad ij\in A_1^-.
\end{array}
\right.
\end{equation}
Therefore, for all
$ij, i'j'\in A_1^+$ we obtain
\begin{align}
(M^t_{kl})^*(\varpi_{ij}\wedge\varpi_{i'j'}+\varpi_{\widehat{ij}}
\wedge \varpi_{\widehat{i'j'}}) & =
\varpi_{ij}\wedge\varpi_{i'j'}+\varpi_{\widehat{ij}}\wedge
\varpi_{\widehat{i'j'}}, \label{eq.2.5} \\
(M^t_{kl})^*(\varpi_{ij}\wedge\varpi_{\widehat{i'j'}}-
\varpi_{\widehat{ij}}\wedge \varpi_{i'j'}) & =\varpi_{ij}\wedge
\varpi_{\widehat{i'j'}}- \varpi_{\widehat{ij}} \wedge
\varpi_{i'j'}. \notag
\end{align}

Consider now the commutative polynomial ring
$RD=\mathbb{R}[x_{ij}; ij\in\overline{D}, i<j]$. Put
$x_{ij}=-x_{ji}$ for $i>j$ and
$x_{ii}=0$. Denote by $RD_I$ the subring of
$RD$ generated by the family of polynomial functions
\begin{align*}
X_I & = \{x_{ij} : ij\in A_0\cup A_2 \} \\ & \quad\; \cup \{
x_{ij}x_{i'j'}+x_{\widehat{ij}}x_{\widehat{i'j'}},\,
x_{ij}x_{\widehat{i'j'}}-x_{\widehat{ij}}x_{i'j'} : ij, i'j'\in
A_1^+\}.
\end{align*}

Since all the 2-forms
$\varpi_{ij}$ commute,
$\Omega^8_0$ is invariant with respect to the one-parameter group
$M^t_{kl}$ if the polynomial function
$F=\sum_{ij,i'j'\in\overline{D}} x_{ij}x_{ij'}x_{i'j}x_{i'j'}$
is an element of the subring
$RD_I$. To prove this fact, note that the sequence
${ij},{ij'},{i'j},{i'j'}\in\overline{D}$ is a sequence of vertices
of either a rectangle or a degenerate
rectangle made of entries of a square
$9\times 9$ matrix without the diagonal.
This sequence originates an either $4$- or
$2$- or $1$-element subset of
$\overline{D}$. So it is natural to consider the following sets:
\begin{align*}
\overline{D}_4 &= \bigl\{ \{ij, ij',i'j, i'j'\}
\subset\overline{D}: i\not=i',\  j\not=j'\bigr\}, \\
\noalign{\smallskip}
\overline{D}_2 &= \bigl\{ \{ij, i'j'\}\subset\overline{D}:
i=i'\ \text{or}\  j=j',\  ij\not=i'j' \bigr\}.
\end{align*}
Using these sets we can rewrite the polynomial
$F$ as a sum
$F=F_1+F_2+F_4$ of three polynomials
\begin{equation}\label{eq.2.6}
F=\sum_{ij\in\overline{D}}x_{ij}^4+
2\sum_{\{ij,i'j'\}\in\overline{D}_2} x_{ij}^2x_{i'j'}^2+
4\sum_{\{ij,ij',i'j, i'j'\}\in\overline{D}_4}
x_{ij}x_{ij'}x_{i'j}x_{i'j'}.
\end{equation}

Consider the polynomial
$F_1+F_2$. Using the decomposition
$\overline{D}=A_0\sqcup A_2\sqcup A_1^+\sqcup A_1^-$,
we can write the first polynomial
$F_1$ as a sum
$F_1= F_{1,0}+F_{1,2}+F_{1,1}^{+}+F_{1,1}^-$ (replacing the set
$\overline{D}$ in the formula for $F_1$ by
$A_0$, $A_2$, $A_1^+$ and
$A_1^-$, respectively). We also consider the decomposition
${\overline{D}}_2=\overline{D}_{2,0}\sqcup\overline{D}_{2,1}\sqcup\overline{D}_{2,2}$
of the set $\overline{D}_2$, where each
$\{ij,i'j'\}\in\overline{D}_{2,\alpha}$ has
$\alpha$ common elements with the subset
$P_{kl}\subset\overline{D}$; and the corresponding decomposition
$F_2=F_{2,0}+F_{2,1}+F_{2,2}$ of
$F_2$.

By definition,
$F_{1,0}+F_{1,2}\in RD_I$. Since
$A_1^{\pm}\subset P_{kl}$, we have
$F_{2,0}\in RD_I$. Taking into account that in any pair
$\{ij,i'j'\}\in\overline{D}_{2,1}$
one element belongs to the subset
$A_0\subset\overline{D}$ and the other to the subset
$A_1^+\sqcup A_1^-$ (i.e. either
$\widehat{ij}=ij$ or
$\widehat{i'j'}=i'j'$), we conclude that
$$
F_{2,1}=2\sum_{\{ij,i'j'\}\in\overline{D}_{2,1}}
x_{ij}^2x_{i'j'}^2=\sum_{\{ij,i'j'\}\in\overline{D}_{2,1}}
(x_{ij}^2x_{i'j'}^2+ x_{\widehat{ij}}^2x_{\widehat{i'j'}}^2)\in
RD_I.
$$
Taking into account that
$x_{ij}^2=x_{ji}^2$ and
$P_{kl}=r^*_k\sqcup c^*_k\sqcup r^*_l\sqcup c^*_l\sqcup A_2$,
we can rewrite the polynomial
$F_{2,2}$ as
$$
4\sum_{j,j'\not\in\{k,l\},
j<j'}(x_{kj}^2x_{kj'}^2+x_{lj}^2x_{lj'}^2)
+4\sum_{j\not\in\{k,l\}}x_{kl}^2(x_{kj}^2+x_{lj}^2)
+4\sum_{j\not\in\{k,l\}}x_{kj}^2x_{lj}^2.
$$
But
$$
F_{1,1}^+ + F_{1,1}^-= \sum_{ij\in
A_1^+}(x_{ij}^4+x_{\widehat{ij}}^4)=
2\sum_{j\not\in\{k,l\}}(x_{kj}^4+x_{lj}^4).
$$
Therefore the component
$$
4\sum_{j\not\in\{k,l\}}x_{kl}^2(x_{kj}^2+x_{lj}^2)
+4\sum_{j\not\in\{k,l\}}x_{kj}^2x_{lj}^2
+2\sum_{j\not\in\{k,l\}}(x_{kj}^4+x_{lj}^4)
$$
of the polynomial
$F_{1,1}^+ +F_{1,1}^-+F_{2,2}$ is an element of
$RD_I$ because
$\widehat{kj}=lj$ and, consequently,
$x_{kl}$,
$(x_{kj}^2+x_{lj}^2)\in RD_I$.

Denote the first term (polynomial) in the above expression of
$F_{2,2}$ as
$F_2^*$. It only remains to be proved that
$F_2^*+F_4\in RD_I$.

For any pair
$ij, i'j'\in\overline{D}$ with
$\{i,j\}\cap\{i',j'\}=\emptyset$ denote by
$q(ij, i'j')$ the quadruple
$\{ij, i'j,ij'\!,i'j'\!\}$. It is clear that
$q(ij,i'j')=q(i'j', ij)$ and
$q(ij, i'j')=q(ij', i'j)$. Moreover, the involution
$\mu$ on $\overline{D}$ induces on the set
$\overline{D}_4$ a well-defined involution
$\mu_4\colon \{ij, i'j, ij', i'j'\} \mapsto\{\widehat{ij},\widehat{i'j},
\widehat{ij'},\widehat{i'j'}\}$
(it is easy to verify that the image of
the rectangle is a rectangle). In particular,
$q(ij,i'j')\mapsto q(\widehat{ij},\widehat{i'j'})$.

Taking into account that the set
$P_{kl}=A_1^+\sqcup A_1^-\sqcup A_2$ is a
union of two rows and two columns, any quadruple
$q\in \overline{D}_4$ has either zero, or
two, or three or four common points with the set
$P_{kl}$. Denote the corresponding subsets of
$\overline{D}_4$ by $\overline{D}_{4,0}$,
$\overline{D}_{4,2}$,
$\overline{D}_{4,3}$,
$\overline{D}_{4,4}$, respectively. Then
$F_4=F_{4,0}+F_{4,2}+F_{4,3}+F_{4,4}$, where the polynomial
$F_{4,\alpha}$ corresponds to the subset
$\overline{D}_{4,\alpha}\subset\overline{D}_4$,
$\alpha=0,2,3,4$. We claim that
$F_{4,0}+F_{4,2}+F_{4,3}\in RD_I$.
To prove this fact, consider the sets
$\overline{D}_{4,0}$,
$\overline{D}_{4,2}$,
$\overline{D}_{4,3}$ in more detail.

If $q\in \overline{D}_{4,0}$ then the four elements of
$q$ belong to the set $A_0$, i.e.\
$F_{2,0}\in RD_I$. If
$q\in \overline{D}_{4,2}$
then two elements of $q$ belong to
$A_0$ and two elements of $q$ belong to
$A_1^+$ or $A_1^-$, i.e.\ the set
$\overline{D}_{4,2}$ is invariant under
the natural action of the involution
$\mu_4$ on
$\overline{D}_{4}$ and a fixed point set for this action on
$\overline{D}_{4,2}$ is empty. Therefore
$$
F_{4,2}=2 \sum_{\{ij, ij', i'j, i'j'\}\in\overline{D}_{4,2}}
(x_{ij}x_{ij'}x_{i'j}x_{i'j'}+ x_{\widehat{ij}}x_{\widehat{ij'}}
x_{\widehat{i'j}}x_{\widehat{i'j'}})\in RD_I.
$$
In the third case, each quadruple
$q\in \overline{D}_{4,3}$ contains
precisely one element of the set $A_2$, so
$\overline{D}_{4,3}=\overline{D}_{4,3}^{kl}\sqcup \overline{D}_{4,3}^{lk}$,
where $\overline{D}_{4,3}^{kl}$ (resp.
$\overline{D}_{4,3}^{lk}$) is the set of all elements from
$\overline{D}_{4,3}$ containing the pair
$kl$ (resp.
$lk$). This decomposition of the set
$\overline{D}_{4,3}$ determines the decomposition
$F_{4,3}=F_{4,3}^{kl}+F_{4,3}^{lk}$ of the polynomial
$F_{4,3}$. Each quadruple
$q\in \overline{D}_{4,3}^{kl}$ is uniquely defined by the pair
$\{kl,ij\}$, i.e.\ by some element
$ij\in\overline{D}$, so
$q=\{kl,kj,il,ij\}$. It is clear that in this element,
$ij\in A_0$ whilst $kj\in A_1^+$ and
$il\in A_1^-$. But for an arbitrary
$ij\in\overline{D}$ the quadruple
$q(kl,ij)=\{kl,kj,il,ij\}$ belongs to
$\overline{D}_4$ if and only if
$\{k,l\}\cap\{i,j\}=\emptyset$. Since this unique relation
defining the quadruples is invariant under interchange of
$i$ and $j$, the quadruple
$q(kl,ij)\in\overline{D}_4$ if and only if
$q(kl,ji)=\{kl,ki,jl,ji\}\in\overline{D}_4$.
Therefore the correspondence
$q(kl,ij)\mapsto q(kl,ji)$ determines
an involutive automorphism on the set
$\overline{D}_{4,3}^{kl}$. Taking into account that
$x_{ji}=-x_{ij}$ and
$q(kl,ij)\not=q(kl,ji)$ we obtain that
\[
F_{4,3}^{kl} = 2
\sum_{\substack{ij\in\overline{D}\\ i,j\not\in\{k,l\}}}
x_{kl}(x_{kj}x_{il}x_{ij}+x_{ki}x_{jl}x_{ji}) =
2\sum_{\substack{ij\in\overline{D}\\ i,j\not\in\{k,l\}}}
x_{kl}x_{ij}(x_{kj}x_{il}-x_{lj}x_{ik}).
\]
Since
$kl, ij\in A_0\sqcup A_2$, $kj\in A_1^+$,
$il\in A_1^-$ and $lj=\widehat{kj}$,
$ik=\widehat{il}$, we have
$F_{4,3}^{kl}\in RD_I$. Similarly,
$F_{4,3}^{lk}\in RD_I$.

If a quadruple
$q\in\overline{D}_4$ has four common points with the set
$P_{kl}$, then either
$q=\{kj, kj', lj, lj'\}$ or
$q=\{jk, j'k, jl, j'l\}$, i.e.\ two elements of
$q$ belong to the short
$k$-row (or column) and another two to the short
$l$-row (or column). Since
$x_{ij}=-x_{ji}$, we have
$F_{4,4}= 8\sum_{0\leqslant j<j'\leqslant 8} x_{kj}x_{kj'}x_{lj}x_{lj'}$,
where
$j,j'\not\in\{k,l\}$, so
$$
F_2^*+F_{4,4}=4\sum_{j,j'\not\in\{k,l\},j<j'}
(x_{kj}^2x_{kj'}^2+x_{lj}^2x_{lj'}^2+
2x_{kj}x_{kj'}x_{lj}x_{lj'})\in RD_I,
$$
because
$(x_{kj}x_{kj'}+x_{lj}x_{lj'})^2\in RD_I$ by definition.

In conclusion, the form
$\Omega^8_0$ is invariant with respect
to the action of each of the subgroups
$M^t_{kl}$ generating the Lie group
$\mathrm{Spin}(9)$, i.e.\ the form
$\Omega^8_0$ is
$\mathrm{Spin}(9)$-invariant.

\medskip

{\it The $8$-form
$\Omega^8_0$ is not trivial.} To
this end we consider the eight vectors
$X_i=(u_i,0)$ and two vectors
$X=(x,0)$ and
$Y=(y,0)$ belonging to the space
$\mathbb{O}^2$. Using the
expressions~\eqref{eq.2.3} for the endomorphisms
$I_i$, we obtain that for
$0\leqslant i,j\leqslant 7$,
$i\not= j$, one has
$$
\varpi_{ij}(X,Y)=g(X, I_{ij} Y)= \inp{(x,0)}{(u_i(\bar u_j y),0)}
= \inp{x}{u_i(\bar u_j y)}
$$
and
\begin{equation}\label{eq.2.7}
\varpi_{i8}(X,Y)=0,
\end{equation}
because the vector
$X=(x,0)$ is orthogonal to
$I_{i8} Y=(0,-\bar y u_i)$. We can rewrite the expression for
$\varpi_{ij}(X,Y)$ as
\begin{equation}\label{eq.2.8}
\varpi_{ij}(X,Y)= \langle x, u_i(\bar u_j
y)\rangle= \langle \bar u_i x,\bar u_j y\rangle= \langle \bar x
u_i,\bar y u_j\rangle,
\end{equation}
because~(cf.\ \cite[Sect.\ 2]{BroGra}) for arbitrary octonions
$a ,b,c\in\mathbb{O}$, one has
\begin{equation}\label{eq.2.9}
\langle a b,c\rangle= \langle
b,\bar a c\rangle= \langle
a ,c\bar b\rangle \quad\text{and}\quad \langle a
,b\rangle= \langle\bar a  ,\bar b\rangle.
\end{equation}

Since $\Omega^8_0$ is a sum of the $8$-forms
$W(i,i';j,j')=\varpi_{ij} \wedge \varpi_{ij'} \wedge \varpi_{i'j}
\wedge\varpi_{i'j'}$, it is sufficient to show that
$W^0(i,i';j,j')=W(i,i';j,j')(X_0,\dotsc,X_7)<0$. It is clear that
the 8-form  $W(i,i';j,j')$ is determined by the unordered pairs
$\{i,i'\}$ and $\{j,j'\}$ of rows and columns, so
$W(i,i';j,j')=W(i',i;j,j')$ and $W(i,i';j,j')=W(i,i';j',j)$.
Moreover, since $\varpi_{ij}=-\varpi_{ji}$ and all these 2-forms
commute, we have
\begin{equation}\label{eq.2.10}
W(i,i';j,j')=W(j,j';i,i').
\end{equation}

Let $S_8$ be the permutation group  acting on the set
$B=\{u_0,\dotsc,u_7\}$ and let $B^\pm=\{\pm u_0,\dotsc,\pm u_7\}$.
For arbitrary $v,v',w,w'\in B^\pm$, put
\[
\widetilde W^0(v,v';w,w')= 2^{-4}\sum_{\sigma\in S_8}A_\sigma(v,v';w,w'),
\]
where $A_\sigma$, for $\sigma=(u_{i_0},\dotsc,u_{i_7})$, is given by
\[
A_\sigma(v,v';w,w') = \varepsilon(\sigma) \langle u_{i_0},
v(w u_{i_1})\rangle
\langle u_{i_2}, v(w' u_{i_3})\rangle
\langle u_{i_4}, v'(w u_{i_5})\rangle
\langle u_{i_6}, v'(w' u_{i_7})\rangle.
\]

As the elements $v,v',w,w'$ occur in this expression twice, we have
\begin{equation}\label{eq.2.11}
\widetilde W^0(v,v';w,w')=
\widetilde W^0(\pm v,\pm v';\pm w,\pm w').
\end{equation}
By definition
$W^0(i,i';j,j')=\widetilde W^0(u_i,u_{i'};\bar u_j,\bar u_{j'})$,
but as
$\bar u_l=\pm u_l$, it follows that
\begin{equation}\label{eq.2.12}
W^0(i,i';j,j')=\widetilde W^0(u_i,u_{i'};u_j,u_{j'}).
\end{equation}

We now prove two lemmas.

\begin{lemma}\label{le.2.2}
For an arbitrary automorphism $\Phi$ of the algebra $\mathbb{O}$
preserving the set $B^\pm$, one has
$\widetilde W^0(v,v';w,w')= \widetilde
W^0(\Phi(v),\Phi(v');\Phi(w),\Phi(w'))$.
\end{lemma}
\begin{proof}
It is clear that
$\Phi(u_k)=\varepsilon^\Phi_{u_k} \sigma^\Phi(u_{k})$, where
$\varepsilon^\Phi_{u_k}=\pm 1$ and
$\sigma^\Phi$ is some permutation in
$S_8$. Moreover, since
$\Phi$ is an element of the exceptional connected Lie group
$\mathrm{G}_2\subset \mathrm{SO}(7)$, we have
$\prod_{k=0}^{7}\varepsilon^\Phi_{u_k}\cdot \varepsilon(\sigma^\Phi)=1$
and, consequently, we have
$A_{\sigma^\Phi\sigma}(\Phi(v),\Phi(v');\Phi(w),\Phi(w'))=
A_{\sigma}(v,v';w,w')$, because
$\varepsilon(\sigma^\Phi\sigma)=\varepsilon(\sigma^\Phi)
\varepsilon(\sigma)$ and
$\sigma^\Phi\sigma(u_k)=\varepsilon^\Phi_{\sigma(u_k)}
\Phi(\sigma(u_k))$. Noting that $\sigma^\Phi S_8=S_8$, we
conclude.
\end{proof}
\begin{lemma}\label{le.2.3}
For any $u\in B^\pm$, one has
 $\widetilde W^0(v,v';w,w')= \widetilde W^0(vu,v'u;wu,w'u)$.
\end{lemma}
\begin{proof}
Since the lemma is obvious for
$u=\pm u_0$, assume that
$u\not= \pm u_0$. Due to the
relations~\eqref{eq.2.8} and the fact that
$\bar u_k=\pm u_k$, we can rewrite the expression for
$A_\sigma(v,v';w,w')$ as
$\varepsilon(\sigma) \langle v u_{i_0},
w u_{i_1}\rangle \langle v u_{i_2},
w' u_{i_3}\rangle \langle v' u_{i_4},
w u_{i_5}\rangle \langle v' u_{i_6}, w' u_{i_7}\rangle$
(the elements $v,v',w,w'$ occur in this expression twice). But for
arbitrary octonions $a ,b,c$, their associator
$(a,b,c)=(ab)c-a(bc)$ is skew-symmetric with
respect to the second and third arguments, i.e.\
$(a b)c+(a c)b= a (bc+cb)$ (cf.~\cite[Sect.\ 2]{BroGra}). Thus, if
$u_k u=-u u_k$ then
$(a u_k)u=(-a u)u_k$. Since
$u\not= \pm u_0$, one has
$u_k u\not=-u u_k$ if and only if either
$u_k=u_0$ or
$u_k=\pm u$. It is clear that in these two cases one has
$(a u)u_k=(a u_k)u$. Noting then
that precisely six elements of the set
$B$ anticommute with
$u$ and that by~\eqref{eq.2.9}, one has
$\langle a u, b u\rangle= \langle a , (b u)\bar u \rangle=
\langle a , b |u|^2 \rangle= \langle a , b\rangle$,
we conclude.
\end{proof}

Suppose now as usual that the basis $B$ coincides with the set
$\{1,\mathbf{i},\mathbf{j},\mathbf{i}\mathbf{j},\mathbf{e},
\mathbf{i}\mathbf{e}$, $\mathbf{j}\mathbf{e},
(\mathbf{i}\mathbf{j})\mathbf{e}\}$, where $\mathbf{i}=u_1$,
$\mathbf{j}=u_2$ and $\mathbf{e}=u_4$, so that for instance
$u_5=u_1 u_4$. Each element of the algebra $\mathbb{O}$ admits a unique
expression as $q_1+q_2 \mathbf{e}$ with $q_1,q_2\in\mathbb{H}$, where
$\mathbb{H}$ is the quaternion algebra generated by
$\mathbf{i},\mathbf{j}$. Then the multiplication in $\mathbb{O}$ is
defined by the standard multiplication relations in $\mathbb{H}$ and by
the relations
\begin{equation}
\label{eq.2.13}
q_1(q_2 \mathbf{e})=(q_2q_1)\mathbf{e}, \quad (q_1
\mathbf{e})q_2=(q_1\bar q_2)\mathbf{e}, \quad (q_1
\mathbf{e})(q_2 \mathbf{e})= -\bar q_2 q_1.
\end{equation}

Put $B^0=B\setminus u_0$. Let $\mathbf{i'},\mathbf{j'},\mathbf{e'}$ be
three arbitrary distinct elements of the set
$B^0\cup(-B^0)$ such that
$\mathbf{e'}\not=\pm\mathbf{i'}\mathbf{j'}$. Then there exists a
unique automorphism $\Phi$ of the octonion algebra $\mathbb{O}$ such that
$\Phi(\mathbf{i'})=u_1$, $\Phi(\mathbf{j'})=u_2$ and
$\Phi(\mathbf{e'})=u_4$ (cf.\ \cite[Lect.\ 15]{Pos}). It is evident
that $\Phi(u_0)=u_0$. Now, taking into account  Lemmas~\ref{le.2.2}
and~\ref{le.2.3}, and the relations~\eqref{eq.2.10}, \eqref{eq.2.11}
and~\eqref{eq.2.12}, we have to calculate only the four numbers
\[
\widetilde W^0(u_0,u_0; u_1, u_1), \;\, \widetilde W^0(u_0,u_0;
u_1, u_2), \;\, \widetilde W^0(u_0,u_1; u_2, u_3), \;\, \widetilde
W^0(u_0,u_1; u_2, u_4).
\]

Indeed, calculating $\widetilde W^0(u_i,u_{i'};u_j,u_{j'})$,
by Lemma~\ref{le.2.3} we can suppose that $u_i=u_0$. If the
sequence $(ij,ij',i'j,ij')$ originates a $1$-element subset of $\overline{D}$,
i.e.\ $i=i'=0$ and $j=j'$, then $\Phi(u_j)=u_1$ for some automorphism
$\Phi$ of $\mathbb{O}$;
if $(ij,ij',i'j,ij')$ originates an 2-element subset of $\overline{D}$,
for instance $i=i'=0$ and $j\not=j'$, then $\Phi(u_j)=u_1$ and
$\Phi(u_{j'})=u_2$ for some automorphism $\Phi$ (when $j=j'$ we can
suppose by Lemma~\ref{le.2.3} that $j=0$ and use~\eqref{eq.2.10}); if this
sequence originates an 4-element subset of $\overline{D}$,
i.e.\ all $i=0,i',j,j'$ are distinct, then according to
either $u_{j'}=\pm u_{i'}u_j$ or $u_{j'}\not=\pm u_{i'}u_j$,
we can obtain as image of the triple $u_{i'};u_j,u_{j'}$ under
$\Phi$ the triple $u_1; u_2, u_3$ or $u_1; u_2, u_4$, respectively.

First of all we consider the restriction $\varpi'_{ij}$ of the
form $\varpi_{ij}$ to the subspace $V\subset\mathbb{O}^2$ generated by
the vectors $X_k$, for $k=0,\dotsc,7$. Let
$\{x^*_0,\dotsc,x^*_7\}$ be the dual basis of $V^*$. Using the
relations~\eqref{eq.2.13} it is easy to verify that
\[
\varpi'_{01}= x^*_0\wedge x^*_1 +x^*_2\wedge x^*_3 +x^*_4\wedge
x^*_5 -x^*_6\wedge x^*_7.
\]
Therefore we have
$\varpi'_{01}\wedge \varpi'_{01}\wedge \varpi'_{01} \wedge
\varpi'_{01}=-24\, x^*_0\wedge x^*_1 \wedge \dotsb \wedge x^*_7$,
that is, $\widetilde W^0(u_0,u_0; u_1, u_1)=-24$. Thus $\widetilde
W^0(u_i,u_i; u_j, u_j)=-24$ for arbitrary $0\leqslant i,j\leqslant
7$, $i\not=j$, because $u_ju_i\not=\pm u_0$ and, consequently, there
exists some automorphism $\Phi$ such
that $\Phi(\pm u_ju_i)=u_1$. In other words,
\[
\varpi'_{ij}= \varepsilon_0 x^*_{i_0}\wedge x^*_{i_1}
+\varepsilon_2 x^*_{i_2}\wedge x^*_{i_3} +\varepsilon_4
x^*_{i_4}\wedge x^*_{i_5} +\varepsilon_6 x^*_{i_6}\wedge
x^*_{i_7},
\]
where $\sigma_{ij}=(i_0,\dotsc,i_7)$ is some permutation of the
set $\{0,\dotsc,7\}$, $\varepsilon_{2k}=\pm 1$, and $\prod_{k=0}^3
\varepsilon_{2k}\cdot \varepsilon(\sigma_{ij})=-1$. Consider also
the form
$$
\varpi'_{ij'}= \varepsilon'_0 x^*_{j_0}\wedge x^*_{j_1}
+\varepsilon'_2 x^*_{j_2}\wedge x^*_{j_3} +\varepsilon'_4
x^*_{j_4}\wedge x^*_{j_5} +\varepsilon'_6 x^*_{j_6}\wedge
x^*_{j_7},
$$
where $i\not=j'$ and $j'\not=j$.

We now show two more lemmas.

\begin{lemma}
For arbitrary distinct elements
$i,j,j'\in\{0,\dotsc,7\}$, the $4$-form $\varpi'_{ij}\wedge
\varpi'_{ij'}$ is a sum of at most eight linearly independent terms
$($$4$-forms$)$ $\varpi'_{k,ij,ij'}$, $k=0,\dotsc,7$, of type $\pm
x^*_{k_0}\wedge x^*_{k_1}\wedge x^*_{k_2}\wedge x^*_{k_3}$. For
each such term $\varpi'_{k,ij,ij'}$, there is a unique term
$\varepsilon_{2p}\, x^*_{i_{2p}}\wedge x^*_{i_{2p+1}}$ of
$\varpi_{ij}$ and a unique term $\varepsilon'_{2p'}
x^*_{j_{2p'}}\wedge x^*_{j_{2p'+1}}$ of $\varpi_{ij'}$ such that
their exterior product is proportional to $\varpi'_{k,ij,ij'}$
$($and, consequently, it is equal to $\varpi'_{k,ij,ij'}$$)$.
\end{lemma}

\begin{proof}
Put $u_l=\pm u_iu_j$ and $u_{l'}=\pm u_iu_{j'}$. It is clear that
$u_l$ and $u_{l'}$ are two distinct
imaginary units of $\mathbb{O}$. Therefore if $\varpi'_{ij}(u_{i_0},
u_{i_1})=\pm \langle  u_{i_0}, u_l u_{i_1}\rangle\not=0$ then
$u_l=\pm u_{i_0}u_{i_1}$ and $u_{l'}\not=\pm u_{i_0}u_{i_1}$,
i.e.\ $\varpi'_{ij'}(u_{i_0}, u_{i_1})=0$. So precisely two terms of
$\varpi'_{ij'}$ contain $x^*_{i_0}$ and $x^*_{i_1}$ as a factor.
Therefore there exists precisely two terms of
$\varpi'_{ij'}$ such that their exterior
product with $x^*_{i_0}\wedge x^*_{i_1}$ is not zero. Since the form
$\varpi'_{ij}$ contains four terms, the number of linearly independent
terms of $\varpi'_{ij} \wedge \varpi'_{ij'}$ is at most eight.

Assume that the product of the terms $x^*_{i_0}\wedge x^*_{i_1}$
and $x^*_{j_0}\wedge x^*_{j_1}$ of the forms $\varpi'_{ij}$ and
$\varpi'_{ij'}$ respectively, is not trivial, i.e.\
$\{i_0,i_1\}\cap\{j_0,j_1\}=\emptyset$. The forms $\varpi'_{ij}$
and $\varpi'_{ij'}$ contain a unique term with the factor
$x^*_{i_0}$. As we show above, in the form $\varpi'_{ij'}$ the
second factor of this term is not equal to $x^*_{i_1}$. Assume
that this factor is equal to $x^*_{j_k}$, $k=0,1$. Then
$\varpi'_{ij'}(u_{i_0}, u_{j_k})\not=0$, i.e.\ $u_{i_0}=\pm
u_{l'}u_{j_k}$. But $u_{j_0}=\pm u_{l'}u_{j_1}$, i.e.\
$\{i_0,i_1\}\cap\{j_0,j_1\}\not=\emptyset$. This contradicts our
non-triviality assumption. We can proceed similarly in the case of
the factor $x^*_{i_1}$.
\end{proof}

\begin{lemma}\label{le.2.5}
For arbitrary distinct elements
$i,j,j'\in\{0,\dotsc,7\}$ and for $0\leqslant i'\leqslant7$, the
expression $\widetilde W^0(u_i,u_{i'};u_j,u_{j'})=
2^{-4}\sum_{\sigma\in S_8}A_\sigma(u_i,u_{i'};u_j,u_{j'})$
contains at most $2^4\cdot 8$ non-zero terms.
\end{lemma}

\begin{proof}
By the previous lemma, each term of $\varpi'_{ij}\wedge
\varpi'_{ij'}$ is the exterior product of a uniquely defined pair
of terms of the forms $\varpi'_{ij}$ and $\varpi'_{ij'}$. On the
other hand, this term of $\varpi'_{ij}\wedge \varpi'_{ij'}$
determines a unique complementary factor in $x^*_0\wedge
\dotsb \wedge x^*_7$ which belongs to
$\varpi'_{i'j}\wedge \varpi'_{i'j'}$.
If such a factor exists, then
$i'\not\in\{j,j'\}$ and by the previous lemma this factor is the
exterior product of a uniquely defined pair of terms of the forms
$\varpi'_{i'j}$ and
$\varpi'_{i'j'}$. Since the number of terms of
$\varpi'_{ij}\wedge \varpi'_{ij'}$ equals at most
$8$ and due to the skew-symmetry of the
$2$-forms, the Lemma follows.
\end{proof}

Suppose that
$i,j,j'\in\{0,\dotsc,7\}$ and
$i',j,j'\in\{0,\dotsc,7\}$ are two
triples containing three distinct elements.
Due to the skew-symmetry of the $2$-forms, one has
$\widetilde W^0(u_i,u_{i'};u_j,u_{j'})= \sum_{[\sigma]\in
S'_8}A_\sigma(u_i,u_{i'};u_j,u_{j'})$,
where $S'_8$ $=S_8/S'$ and the subgroup $S'\subset S_8$ is generated
by the $4$ transpositions $(0,1)$, $(2,3)$, $(4,5)$, and $(6,7)$.
By Lemma~\ref{le.2.5} this sum contains at most $8$ non-zero
terms. Let us describe these terms. To this end,
using~\eqref{eq.2.8} we can rewrite the expression for
$A_\sigma(v,v';w,w')$ as
$$
-\varepsilon(\sigma) \langle  u_{i_0} v, u_{i_1} w \rangle \langle
u_{i_2} v, u_{i_3} w' \rangle \langle u_{i_4} v', u_{i_5} w
\rangle \langle u_{i_6} v', u_{i_7} w' \rangle,
$$
as $\bar u_k=-u_k$ for all of the seven imaginary units
and the elements $v,v',w,w'$ occur in this expression twice. Let
$u\in B$ and $a \in B^\pm$. Applying  the same
arguments as in the proof of
Lemma~\ref{le.2.3}, we obtain that if $a  u=-u a $ then
$(u_ka )u=(-u_k u)a $. But $a  u\not=-u a $ if and
only if $a =\pm u$ or $a =\pm u_0$ or $u=\pm u_0$. In all
these cases $(u_ka )u=(u_k u)a $. Since $\langle a  u,
b u\rangle= \langle a , b\rangle$, we obtain the
following  expression for $A_\sigma(v,v';w,w')$: {\small
$$
-\varepsilon(\sigma) \bigl\langle  (u_{i_0}u) v, (u_{i_1}u)
w\bigr\rangle
 \bigl\langle (u_{i_2}u) v, (u_{i_3}u) w'\bigr\rangle
 \bigl\langle (u_{i_4}u) v'\!, (u_{i_5}u) w\bigr\rangle
 \bigl\langle (u_{i_6}u) v', (u_{i_7}u) w'\bigr\rangle
$$}
\noindent (the elements $v,v',w,w'$ occur in this expression
twice).

Suppose now that $A_\sigma(u_i,u_{i'};u_j,u_{j'})\not=0$ for some
$\sigma\in S_8$.  Right multiplication by $u$ determines the
permutation $\sigma^u$ of the set $B$: $u_k
u=\varepsilon^u_{u_k}\sigma^u(u_k)$ ($\varepsilon^u_{u_k}=\pm 1$).
This permutation is even since if $u\not=u_0$ then $u^2=-u_0$ and
$\sigma^u$ is a product of four independent transpositions.
The sequence $(\varepsilon^u_{u_0},\dotsc,\varepsilon^u_{u_7})$
contains an even number of $-1$. One can easily verify this fact
for $u=u_1$ using~\eqref{eq.2.13} and for the other imaginary
units $u_l$ using an automorphism $\Phi$ for which
$\Phi(u_1)=u_l$:
\begin{gather*}
\Phi(u_k)\Phi(u_1) = \varepsilon^\Phi_{u_k}\sigma^\Phi(u_k)\cdot
u_l=\varepsilon^\Phi_{u_k}\varepsilon^{u_l}_{\sigma^\Phi(u_k)}
\sigma^{u_l}(\sigma^\Phi(u_k)), \\
\noalign{\smallskip}
\Phi(u_k u_1) = \Phi(\varepsilon^{u_1}_{u_k}\sigma^{u_1}(u_k))=
\varepsilon^{u_1}_{u_k}
\varepsilon^\Phi_{\sigma^{u_1}(u_k)}\sigma^\Phi(\sigma^{u_1}(u_k)).
\end{gather*}
Taking into account that
$\prod_{k=0}^7 \varepsilon^\Phi_{u_k}=
\prod_{k=0}^7 \varepsilon^\Phi_{\sigma^{u_1}(u_k)}$,
we have
\[
\prod_{k=0}^7 \varepsilon^{u_l}_{u_k}=
\prod_{k=0}^7 \varepsilon^{u_l}_{\sigma^\Phi(u_k)}=
\prod_{k=0}^7 \varepsilon^{u_1}_{u_k}.
\]
Thus $A_\sigma(u_i,u_{i'};u_j,u_{j'})=A_{\sigma^{u_k}\sigma}
(u_i,u_{i'};u_j,u_{j'})$ for all of the eight even permutations
$\sigma^{u_k}$, $k=0,\dotsc,7$. It only remains to be proved that
the permutations $\sigma^{u_k}\sigma$ determine distinct classes in
the quotient group $S'_8$.

Suppose that $\sigma^{u_k}\sigma=\sigma^{u_p}\sigma\cdot s$
for some element $s\in S_8'$ and $k\not=p$. Taking into account that
$\sigma^{u_p}\sigma^{u_k}=\sigma^{u_k}\sigma^{u_p}=\sigma^{u_q}$,
where $u_q\in B$ and
$u_q=\pm u_ku_p=\pm u_pu_k$, we can assume that $u_p=u_0$ and
$\sigma(u_0)=u_0$. But for $u\in B$ we have
$\{\pm u_0u,\pm u_{i_1}u\}=\{\pm u_0,\pm u_{i_1}\}$
if and only if $u\in\{u_0,u_{i_1}\}$.
Since $A_\sigma(u_i,u_{i'};u_j,u_{j'})\not=0$,
we have $u_{i_1}=u_l$ and
$u_{i_3}=\pm u_{l'}u_{i_2}$, where $u_l=\pm u_iu_j$ and
$u_{l'}=\pm u_iu_{j'}$. Taking into account that $u_l\not=u_{l'}$,
we obtain that $u_{i_3}\not=\pm u_{l}u_{i_2}=\pm u_{i_1}u_{i_2}$,
i.e.\ $u_k=u_0$, a contradiction.
Thus the permutations $\sigma^{u_k}\sigma$ determine 8 distinct classes in
$S'_8$. So if the sequences $(i,j,j')$ and $(i',j,j')$ contain $3$
distinct elements then $W^0(i,i';j,j')=8A_\sigma(u_i,u_{i'};u_j,u_{j'})$,
where $\sigma\in S_8$ is an arbitrary permutation such that
$A_\sigma(u_i,u_{i'};u_j,u_{j'})\not=0$. Using now the
relations~\eqref{eq.2.13}, we can describe such permutations for
the following sequences $(i,i';j,j')$:
\begin{align*}
(0,0; 1,2): & \quad \sigma=(0,1,4,6,2,3,5,7),
\quad\varepsilon(\sigma)=-1, \\
(0,1; 2,3): & \quad  \sigma=(0,2,4,7,5,6,1,3),
\quad\varepsilon(\sigma)=-1, \\
(0,1; 2,4): & \quad  \sigma=(0,2,1,5,4,7,3,6),
\quad\varepsilon(\sigma)=1.
\end{align*}
For all these cases $A_\sigma(u_i,u_{i'};u_j,u_{j'})=-1$.
Thus, if the sequences $i,j,j'$ and $i',j,j'$ or the sequences $i,i',j$ and
$i,i',j'$ from the set $\{0,\dotsc,7\}$  contain three distinct elements
(i.e.\ a sequence $ij,ij',i'j,i'j'$ generates either a rectangle or
an interval) then $W^0(i,i';j,j')=-8$. We also proved that
$W^0(i,i;j,j)=-24$ for all $i,j=0,\dots,7$, $i\not=j$.

Let $\overline{D}'$, $\overline{D}'_2$ and
$\overline{D}'_4$ be sets defined for the index set
$\{0,\dotsc,7\}$ as $\overline{D}$, $\overline{D}_2$ and
$\overline{D}_4$ were defined for the index set
$\{0,\dotsc,8\}$. Then
$\overline{D}'_2\subset\overline{D}_2 $ and
$\overline{D}'_4\subset\overline{D}_4 $. Taking into account that
$\#(\overline{D}'_2)=(8\cdot7)(6\cdot2)/2$ and
$\#(\overline{D}'_4)=(8\cdot7)(6\cdot5)/4$ (for each pair
$ij\in\overline{D}'$ there exist
$6\cdot5$ ordered pairs $i'j'\in\overline{D}'$ such that
$\{i,j\}\cap\{i',j'\}=\emptyset$), from~\eqref{eq.2.6} it follows
that
\[
\Omega^8_0(X_0,\dotsc,X_7) = - 24(8\cdot 7) -
8(8\cdot7\cdot 12) - 8(8\cdot7\cdot30)
= - 14\cdot 1440,
\]
hence $\Omega^8_0$ is not trivial.

We must finally prove that the canonical $8$-form on any
$\mathrm{Spin}(9)$-manifold $(M^{16},g,\nu^9)$, given in the
statement, is globally defined. In other words, we must prove that
the definition of the form $\Omega^8_0$ is independent of the
choice of the basis $\{I_j\}$ of the space $V^9=\nu^9(p)$, $p\in
M$, satisfying the relations~\eqref{eq.2.3}. Indeed, given one such
basis $\{I_j\}$, any other basis $\{I'_j\}$ is obtained as
$I_i^\prime = \sum_{0\leqslant j\leqslant 8}m^j_i I_j$, for
$i=0,\dotsc,8$, and $(m^j_i) \in \mathrm{SO}(9)$. From this fact it
follows in particular that the $\mathrm{Spin}(9)$-groups associated
with these two bases coincide. But as we remarked above,
$\pi(\mathrm{Spin}(9))= \mathrm{SO}(9)=SO(V^9)$, i.e. there exists
some element $s\in\mathrm{Spin}(9)$ such that $sI_js^{-1}=I'_j$,
for all $j=0,\ldots,8$. Now since the group $\mathrm{Spin}(9)$
preserves the scalar product $g_p=\inp{\cdot}{\cdot}$ on $T_pM
\equiv \mathbb{O}^2$ and the form $\Omega^8_0$ is
$\mathrm{Spin}(9)$-invariant, the form $\Omega^8_0$ does not depend
on the chosen basis $\{I_j\}$.

\subsection{Some Corollaries to Theorem \ref{th.1.1}}

We can get some consequences of the proof of Theorem~\ref{th.1.1}.
By~\eqref{eq.2.5} with $ij=i'j'\in A_1^+$, the $4$-form
$\sum_{i,j=0,\dots,8}\varpi_{ij}\wedge\varpi_{ij}$
on the space $T_pM \equiv \mathbb{O}^2$ is invariant with respect to
the action of each of the subgroups $M^t_{kl}$ generating the
Lie group $\mathrm{Spin}(9)$. It is $\mathrm{Spin}(9)$-invariant
hence trivial~(\cite[Sect.\ 5]{BroGra})
so it defines a global (trivial) $4$-form on $M$.
We thus obtain the next corollary to Theorem \ref{th.1.1}.

\begin{corollary}\label{co.2.6}
The $4$-form
$\sum_{0\leqslant i<j\leqslant 8}\omega_{ij}\wedge\omega_{ij}=0$,
vanishes, i.e.\ we have
\begin{multline*}
\sum_{0\leqslant i<j\leqslant 8} \bigl\{ \omega_{ij}(X,Y)\omega_{ij}(Z,W)
                  -\omega_{ij}(X,Z)\omega_{ij}(Y,W) \\
                     \noalign{\vspace{-3mm}}
                          + \omega_{ij}(Y,Z)\omega_{ij}(X,W) \bigr\} =0,
\end{multline*}
or, equivalently,
\begin{equation}\label{eq.2.14}
\Cyclic_{XYZ} \sum_{0\leqslant i<j\leqslant 8} \omega_{ij}(X,Y)W^\flat(I_{ij}Z) = 0,
\qquad X,Y,Z,W\in \mathfrak{X}(M).
\end{equation}
\end{corollary}

Moreover, since the $8$-form $(\sum_{i,j=0,\dots,8}
\varpi_{ij}\wedge\varpi_{ij})\wedge (\sum_{i',j'=0,\dots,8}
\varpi_{i'j'}\wedge\varpi_{i'j'})$ vanishes, we
can rewrite the expression of the canonical form as
\begin{corollary}
\[
\Omega^8 = -\frac12\sum_{\substack{i,j=0,\dots,8\\
i',j'=0,\dots,8}}
(\omega_{ij}\wedge\omega_{i'j'}- \omega_{i'j} \wedge \omega_{ij'})\wedge
(\omega_{ij}\wedge\omega_{i'j'}- \omega_{i'j} \wedge \omega_{ij'}).
\]
\end{corollary}

Furthermore, given a triple $ijq$, we denote by $\widehat{ijq}$ the
new triple obtained by replacing the element $k$ (if it occurs in
$ijq$) by $l$ and the element $l$ (if it occurs in $ijq$) by $k$.
It is easy to verify that for the restriction $\bar\sigma_{ijq}
=\sigma_{ijq}|T_pM$, one has
$$
(M^t_{kl})^*\bar\sigma_{ijq}= \left\{\begin{array}{ll}\bar\sigma_{ijq},
&\text{if}\quad \{k,l\}\cap\{i,j,q\}=\emptyset,\\
\noalign{\smallskip}
\bar\sigma_{ijq}, &\text{if}\quad \{k,l\}\subset\{i,j,q\},
\\\noalign{\smallskip}
\cos 2t\cdot \bar\sigma_{ijq}+\sin 2t\cdot  \bar\sigma_{\widehat{ijq}},
          &\text{if}\quad \{k,l\}\cap\{i,j,q\}=\{k\},\\
\noalign{\smallskip}
\cos 2t\cdot  \bar\sigma_{ijq}-\sin 2t\cdot \bar\sigma_{\widehat{ijq}},
             &\text{if}\quad \{k,l\}\cap\{i,j,q\}=\{l\},
\end{array}\right.
$$
and, consequently, the $4$-form $\sum_{i,j,q=0,\dots,8}
\bar\sigma_{ijq}\wedge\bar\sigma_{ijq}$ on the space $T_pM
\equiv \mathbb{O}^2$ is invariant with respect to the action of each of
the subgroups
$M^t_{kl}$ generating the Lie group $\mathrm{Spin}(9)$. It is
$\mathrm{Spin}(9)$-invariant and, consequently, it is also
trivial~(\cite[Sect.\ 5]{BroGra}), so we obtain

\begin{corollary}
The $4$-form
$\sum_{0\leqslant i<j<k\leqslant 8}\sigma_{ijk}
\wedge\sigma_{ijk}$, vanishes, i.e.\ we have
\begin{multline*}
\sum_{0\leqslant i<j<k\leqslant 8}
\bigl\{\sigma_{ijk}(X,Y)\sigma_{ijk}(Z,W)
                  -\sigma_{ijk}(X,Z)\sigma_{ijk}(Y,W) \\
                     \noalign{\vspace{-3mm}}
                      +\sigma_{ijk}(Y,Z)\sigma_{ijk}(X,W) \bigr\} = 0.
\end{multline*}
\end{corollary}

\begin{remark}
Using the method of the proof of Theorem~\ref{th.1.1} one could
obtain the expression for the canonical form $\Omega^8$ in terms of
the $2$-forms $\sigma_{ijp}$. But since the proof is technically
more complicated, we state it as the next

\smallskip

{\bf Conjecture.} {\it The canonical $8$-form $\Omega^8$
on the $\mathrm{Spin}(9)$-manifold $(M,g,\nu^9)$ is given by}
\begin{equation*}
\Omega^8  = \frac14 \sum_{\substack{i,j=0,\dots,8\\
i',j'=0,\dots,8}}
\sum_{\substack{p,p'=0,\dots,8}}
\sigma_{ijp}\wedge\sigma_{ijp'}\wedge\sigma_{i'j'p}
\wedge\sigma_{i'j'p'}.
\end{equation*}
\end{remark}

\section{$\mathrm{Spin}(9)$-structures as $G$-structures defined by a tensor}
\setcounter{equation}{0}

The concept  of {\it $G$-structure defined\/}
(or {\it characterized$)$ by a tensor\/}
is well known (see Bernard \cite[pp.\
210--212]{Ber}, Fujimoto \cite[p.\ 24]{Fuj}, Mar\'\i n and de
Le\'on \cite[p.\ 377]{MarLeo}, and Salamon \cite[p.\ 11]{Sal};
cf.\ also \cite[pp.\ 127, 175]{Sal}). We now focus our attention
to the case where $G=\mathrm{Spin}(9)$.

We would like to remark firstly that
in this case the tensor used to define a
$\mathrm{Spin}(9)$-structure will never be a stable tensor (cf.\
Friedrich \cite[p.\ 2]{Fri1}, \cite[p.\ 2]{Fri3}). A tensor on
$\mathbb{R}^n$ is said to be stable
if its orbit under the action of
$\mathrm{GL}(n,\mathbb{R})$ is an open subset (see Hitchin
\cite[p.\ 2]{Hit}, Witt \cite[\S\S3.2]{Wit}). These special
structures play an interesting
role in the theory of
$G$-structures. But for
$G=\mathrm{Spin}(9)$ a simple computation of dimensions shows that
the interior of any orbit on the space of
$8$-forms is void.

On the other hand, Friedrich's local bases $\{ \omega_{ij},
\sigma_{ijk} \}$ of $\Lambda^2M$ given in Section~$1$ are related
to the decomposition of $\Lambda^2(\Delta_9)$, which we now recall
(cf.\ e.g.\ Adams \cite[Th.\ 4.6, (ii)]{Ada}). Let $\lambda^r$
denote the representation arising from the $r$th exterior power
representation of $\mathrm{SO}(9)$ via the homomorphism
$\pi\colon\mathrm{Spin}(9)\to\mathrm{SO}(9)$. Then one has
$\Delta_9 \otimes \Delta_9 = \sum_{r=0}^4 \lambda^r$. Moreover, as
$\Delta_9$ is self-dual, we have the decomposition of $\Delta_9
\otimes \Delta_9 \cong \Delta_9^* \otimes \Delta_9 \cong
\mathfrak{gl}(\mathbb{R},16)$ into symmetric and skew-symmetric
components,
\begin{equation}\label{eq.3.1}
S^2(\Delta_9) = \lambda^0 \oplus \lambda^1 \oplus \lambda^4,
\quad \Lambda^2(\Delta_9) = \lambda^2 \oplus \lambda^3,
\end{equation}
where $\lambda^0$ is the center of $\mathfrak{gl}(16,\mathbb{R})$.

We have proved in Theorem \ref{th.1.1} that $\Omega^8_0$ is
$\mathrm{Spin}(9)$-invariant and non-trivial. We now prove that
$\rho(\mathrm{Spin}(9))\subset \mathrm{GL}(16,\mathbb{R})$
is actually the stabilizer group of $\Omega^8_0$
in the group $\mathrm{GL}(16,\mathbb{R})$,
showing that this group is no bigger than
$\rho(\mathrm{Spin}(9))$.

We have

\begin{theorem}\label{th.3.1}
The stabilizer group of the canonical $8$-form $\Omega^8_0$ on
$\mathbb{R}^{16}$, under the natural action of the group
$\mathrm{GL}(16,\mathbb{R})$, is the Lie group $\rho(\mathrm{Spin}(9))$.
\end{theorem}
\begin{proof}
To simplify notation in this proof, we will write simply
$\mathrm{Spin}(9)$ and $\mathfrak{spin}(9)$
instead of $\rho(\mathrm{Spin}(9))$ and $\rho_*(\mathfrak{spin}(9))$,
respectively.
Let $G$ be the stabilizer group of $\Omega^8_0$ and $\mathfrak g$ its
Lie algebra.

As $\mathfrak{spin}(9)$ is a subalgebra of
$\mathfrak{gl}(16,\mathbb{R})$, the adjoint representation of
$\mathfrak{gl}(16,\mathbb{R})$ induces the representation of
$\mathfrak{spin}(9)$ on $\mathfrak{gl}(16,\mathbb{R})$. The set
$\{I_{i_1 \ldots i_r}, 0\leqslant i_1< \dotsb <i_r\leqslant 8\}$ is
a basis of the $\mathfrak{spin}(9)$-invariant subspace $\lambda^r$
of $\mathfrak{gl}(16,\mathbb{R})$ in \eqref{eq.3.1}, for
$r=1,\dotsc,4$, respectively. Moreover, all the operators in each
$\lambda^r$ are traceless (for example, $2I_{i_1i_2i_3i_4}$
$=[I_{i_1},I_{i_2i_3i_4}]$). As the submodules in  \eqref{eq.3.1}
are mutually not isomorphic, if
$\mathfrak{g}\not=\mathfrak{spin}(9)$, then
$\lambda^r\subset\mathfrak{g}$ for some $0\leqslant r\leqslant 4$.
We know that $\mathfrak{so}(16)=\lambda^2 \oplus \lambda^3$ and
$\mathfrak{spin}(9)=\lambda^2$ and it is clear that
$\lambda^0\not\subset\mathfrak{g}$.

Suppose then that $\lambda^1\subset\mathfrak{g}$. Then the one-parameter
subgroup
\[
M^t_8=\cosh t \cdot \mathrm{I} + \sinh t \cdot I_8
\subset \mathrm{GL}(16,\mathbb{R})
\]
generated by the vector $I_8\in\mathfrak{gl}(16,\mathbb{R})$, would be a
subgroup of $G$. It is easy to verify (see the proof
of~\eqref{eq.2.4}) that for any $0\leqslant i<j\leqslant 8$,
\begin{equation*}
(M^t_8)^*\varpi_{ij}=
\left\{
\begin{array}{ll}
\varpi_{ij},  &\text{if}\quad 8\in\{i,j\},\\
\noalign{\smallskip}
\cosh 2t\cdot \varpi_{ij}+\sinh 2t\cdot\bar\sigma_{ij8},
&\text{if}\quad i,j<8.
\end{array}
\right.
\end{equation*}
Let $V\subset\mathbb{O}^2$ be (as in the proof of Theorem \ref{th.1.1})
the subspace with basis $X_i=(u_i,0)$, $i=0,\dotsc,7$.
Then by~\eqref{eq.2.7}, we have $\varpi_{i8}|V=0$.
Further,
$\bar\sigma_{ij8}|V=-\varpi_{ij}|V$,
because by~\eqref{eq.2.3} one has
$I_8v=-v$ for all
$v\in V$. Now taking into account the expression for the
$8$-form
$\Omega^8_0$, we obtain that
$$
((M^t_8)^*\Omega^8_0)|V
= \sum_{\substack{0\leqslant i,j\leqslant 7\\ 0\leqslant i',j'\leqslant7}}
(\cosh 2t-\sinh 2t)^4 (\varpi_{ij} \wedge\varpi_{ij'}\wedge \varpi_{i'j}
\wedge \varpi_{i'j'})|V,
$$
i.e.\ $((M^t_8)^*\Omega^8_0)|V=(\cosh 2t-\sinh 2t)^4 \Omega^8_0|V$.
Thus $\lambda^1\not\subset\mathfrak{g}$, because $\Omega^8_0|V\not=0$.

The form $\Omega^8_0$ is not $\mathrm{SO}(16)$-invariant. In the
opposite case, it would determine a non-trivial
$\mathrm{SO}(17)$-invariant harmonic differential $8$-form on the
$16$-dimensional sphere $S^{16}$, but since
$H^8(S^{16},\mathbb{R})=0$, we would get a contradiction. Hence
$\lambda^3\not\subset\mathfrak{g}$.

So if $\mathfrak{g}\not=\mathfrak{spin}(9)$ then
$\mathfrak{g}=\lambda^4\oplus\mathfrak{spin}(9)$. It is clear that
$[\lambda^4,\lambda^4]\subset\mathfrak{so}(16)$ and, consequently, the
subspace $\lambda^4\oplus\mathfrak{spin}(9)$ is a Lie algebra if
and only if $[\lambda^4,\lambda^4] \subset\mathfrak{spin}(9)$. But
since $[I_k I_{i_1i_2i_3},I_k
I_{j_1j_2j_3}]=-[I_{i_1i_2i_3},I_{j_1j_2j_3}]$ for any $4$-element
subsets $\{k,i_1,i_2,i_3\}$ and $\{k,j_1,j_2,j_3\}$ of the set
$\{0,\dotsc,8\}$, we have
$[\lambda^3,\lambda^3]\subset[\lambda^4,\lambda^4]$. As the
homogeneous space $\mathrm{SO}(16)/\mathrm{Spin}(9)$ is not a
symmetric space (cf.\ Helgason \cite[p.\ 518]{Hel}), i.e.\
$[\lambda^3,\lambda^3]\not\subset \mathfrak{spin}(9)=\lambda^2$, we
obtain that $[\lambda^4,\lambda^4]\not\subset\mathfrak{spin}(9)$,
that is, $\mathfrak{g}=\mathfrak{spin}(9)$.

It only remains to be proved that the group $G$ is connected. To
this end, similarly to Brown and Gray in~\cite[Prop.\
5.3]{BroGra}, we shall find the normalizer (containing
$G$) of the group $\mathrm{Spin}(9)$ in
$\mathrm{GL}(16,\mathbb{R})$. Suppose that
$A\in \mathrm{GL}(16,\mathbb{R})$ normalizes
$\mathrm{Spin}(9)$. Since
$\mathrm{Spin}(9)$ has no outer
automorphisms there exists an element
$B\in\mathrm{Spin}(9)$ such that
$AB^{-1}$ is in the centralizer in
$\mathrm{GL}(16,\mathbb{R})$ of
$\mathrm{Spin}(9)$. The complexification of the
$16$-dimensional representation of
$\mathrm{Spin}(9)$ is irreducible so
$AB^{-1}$ is a scalar operator $tI$, $t\in\mathbb{R}$. But the operator
$tB$ preserves the 8-form if and only if $t^8=1$.
Since by definition
$\mathrm{Spin}(9)$ contains
$I_1I_2I_1I_2=-I$, we have
$G = \mathrm{Spin}(9)$.
This completes the proof.
\end{proof}

As a consequence of Theorem \ref{th.3.1} we have

\begin{corollary}
A reduction of the structure group of the bundle
of oriented orthonormal frames of a connected, oriented
$16$-dimensional Riemannian manifold $M$ to
$\mathrm{Spin}(9)$ is characterized by a parallel
$8$-form
$\Omega^8$ which is linearly equivalent at each point
$p \in M$ to the
$\mathrm{Spin}(9)$-invariant
$8$-form
$\Omega^8_0$ on $\mathbb{R}^{16}$.
\end{corollary}
\begin{proof}
According to \cite[Props.\ $5.2$, $5.4$, $5.5$]{BroGra} we must
only prove that $\Omega^8_0$ is $\mathrm{Spin}(9)$-invariant but not
$\mathrm{SO}(16)$-invariant. We have proved the first fact in Theorem
\ref{th.1.1} and the second one in the proof of
Theorem~\ref{th.3.1}.
\end{proof}

\section{The curvature tensor of the Cayley planes}
\setcounter{equation}{0}

We now apply our previous conclusions to obtain an expression of
the curvature tensor of the Cayley planes in terms of the nine
local symmetric involutions involved and then to relate it to the
well-known expression in terms of triality given by Brown and Gray
\cite{BroGra}, to the one in terms of the brackets of the Lie
algebra $\mathfrak{f}_4$ of $F_4$, furnished by Brada and
P\'ecaut-Tison \cite{BraPec0,BraPec}, and also to the expression
given in \cite{Myk2}.

First recall (\cite{Ale,BroGra}) that the curvature tensor
$R$ of a non-flat
$\mathrm{Spin}(9)$-manifold is a
non-zero multiple of the curvature tensor
$R^{\operatorname{\mathbb O \mathrm P}(2)}$ of
$\operatorname{\mathbb O \mathrm P}(2)$. Further, as duality
reverses curvature, in the next formulas we can take a constant
$c \in \mathbb{R} \backslash \{0\}$, being understood that
$c>0$ (resp.\ $c<0$) in
the compact (resp.\ noncompact) case.

Then we have
\begin{proposition}
The curvature tensor $R_{XY}Z$ of
the Cayley planes is given by
\begin{equation}
\label{eq.4.1}
R_{XY}Z =-\frac{c}4\sum_{0 \leqslant i < j \leqslant 8}
\omega_{ij}(X,Y)I_{ij}Z,
\qquad c \in \mathbb{R} \backslash \{0\}.
\end{equation}
\end{proposition}

\begin{proof}
The form $\lambda\sum_{i<j} \omega_{ij} \otimes I_{ij}$,
$\lambda\in{\mathbb R}$, is a
$\rho_*(\mathfrak{spin}(9))$-valued $2$-form.
Moreover, the necessary algebraic conditions are clearly satisfied by
$\lambda\sum_{i<j} \omega_{ij} \otimes \omega_{ij}$,
except for the Bianchi identity, but this is immediate from
equation \eqref{eq.2.14}.

As the curvature tensor is a non-zero multiple of
$R^{\operatorname{\mathbb O \mathrm P}(2)}$, it only rests to find
the coefficient of the right-hand side of \eqref{eq.4.1}. To
compute the sectional curvature we take two orthonormal vectors
$v=(x_1,x_2),w=(y_1,y_2) \in S^{15} \subset T_pM \equiv
\mathbb{O}^2$. Now, the map $(v,w)\mapsto-\lambda\sum_{0 \leqslant
i<j \leqslant 8} \omega_{ij}^2(v,w) $ is easily seen
from~\eqref{eq.2.4} to be invariant under each endomorphism
$M^t_{kl}$, hence under $\mathrm{Spin}(9)$. Consider then the
orthonormal basis $e_1=(u_0,0),\dotsc,$ $e_8=(u_7,0),e_9=(0,u_0)$,
$\dotsc,e_{16}=(0,u_7)$ of $\mathbb{O}^2 \equiv T_pM$. As
$\mathrm{Spin}(9)$ acts transitively on $S^{15}$, there exists an
element of $\mathrm{Spin}(9)$ mapping $v$ to $(u_0,0)$ and $w$ to a
vector $w'= \sum_{k=0}^7 \bigl(\mu_k(u_k,0) + \nu_k(0,u_k)\bigr)$
with $\mu_0=0$. So for certain $\lambda \in \mathbb{R}\backslash
\{0\}$, as a computation using~\eqref{eq.2.2} and \eqref{eq.2.3}
shows, we have for $R_{vwvw}=g(R_{vw}v,w)$ that
\[
R_{vwvw} = - \lambda\sum_{\substack{0\leqslant
i<j\leqslant 8\\ 0\leqslant k \leqslant 7}}
\bigl\langle (u_0,0), I_{ij}
\bigl( \mu_k (u_k,0) + \nu_k (0,u_k)\bigr) \bigr\rangle^2
=-\lambda\Bigl( 3\sum_{k=0}^7 \mu_k^2+1 \Bigr).
\]
In fact, the operator $I_{ij}$ acts on the basis
$\{(u_k,0),(0,u_k),k=0,\ldots, 7\}$ as a permutation (up to sign) and
for each vector $(u_k,0)$, $k\geqslant1$, there exist precisely
four different pairs $\{u_i,u_j\}$ for which $u_i(u_ju_k)
\stackrel{\mp}{=} (u_iu_j)u_k\stackrel{\mp}{=}u_0$ and for each
vector $(0,u_k)$, $k\geqslant0$ there exists a unique pair
$\{u_i,u_8\}$ for which $u_iu_k\stackrel{\mp}{=}u_0$ ($i=k$ in this
case), where $\stackrel{\mp}{=}$ means ``equal up to sign.''

Taking $\lambda=-\frac c4$, we see that the
absolute value of the sectional
curvature belongs to $[|c|/4,|c|]$.
\end{proof}

Brown and Gray give in~\cite[(6.12)]{BroGra} an explicit
expression for the curvature tensor
$R_{XY}Z$ of $\operatorname{\mathbb O \mathrm P}(2)$.

Letting
$\mathbb{R}^{16} \equiv \mathbb{O}^2$, according to Lemma 3.1 and formulas
$(4.1)$, $(4.2)$, and $(6.2)$ in their paper, and only changing
some notations, Brown and Gray's formula for the curvature tensor
can be written as $R_{XY}Z=S_{XY}Z-S_{YX}Z$, where
\begin{equation}\label{eq.4.2}
\begin{split}
S_{XY}Z
=-\frac{c}4\bigl(&4\langle y_1,z_1\rangle x_1
+(z_1y_2)\bar x_2+(x_1y_2)\bar z_2, \\
&4\langle y_2,z_2\rangle x_2+
\bar x_1(y_1z_2)+\bar z_1(y_1x_2) \bigr),
\end{split}
\end{equation}
for  $X=(x_1,x_2)$, $Y=(y_1,y_2)$, $Z=(z_1,z_2)\in \mathbb{O}^2$.

They also comment that an expression `similar' to the well-known
ones for the spaces of constant either holomorphic or quaternionic
sectional curvature cannot be given, because, differently to
$\mathrm{U}(n)$ and
$\mathrm{Sp}(n)\mathrm{Sp}(1)$, the group
$\mathrm{Spin}(9)$ has not proper normal subgroups.

However, in \cite[Prop.\ 4]{Myk2} a simple expression for either
$R^{\operatorname{\mathbb O \mathrm P}(2)}$ or
$R^{\operatorname{\mathbb O \mathrm H}(2)}$ has been given in
terms of the nine local symmetric
operators. We can write it as
$R_{XY}Z=S^\prime_{XY}Z-S^\prime_{YX}Z$, where
\begin{equation}\label{eq.4.3}
S^{\prime}_{XY}Z =-\frac{c}4\Bigl(
3g(Y,Z)X + \sum_{0\leqslant i\leqslant 8} g(I_iY,Z)I_iX \Bigr),
\end{equation}
respectively.

This expression, in terms of the octonion algebra
has the following form (see~\cite[Prop.\ 4,(15)]{Myk2}) for
$X=(x_1,x_2)$, $Y=(y_1,y_2)$ and $Z=(z_1,z_2)$,
\begin{equation}\label{eq.4.4}
\begin{split}
S^\prime_{XY}Z=-\frac{c}4
\bigl(&(x_1{\bar y_1})z_1+(x_1y_2){\bar z_2}
+(z_1{\bar y_1})x_1 + (z_1y_2){\bar x_2}, \\
&\; {\bar z_1}(y_1x_2)+z_2({\bar y_2}x_2)
+x_2({\bar y_2}z_2)+{\bar x_1}(y_1z_2)\bigr).
\end{split}
\end{equation}

Using the well-known octonion identities
$\langle x,y\rangle=\langle\bar x,\bar y\rangle$ and
$2\langle x,y\rangle a= (ax)\bar y+(ay)\bar x$ and
their conjugated $2\langle x,y\rangle a= y(\bar xa)+x(\bar ya)$
for arbitrary $x,y,a\in \mathbb{O}$
(see~\cite[Lect. 15, (1)]{Pos}), we obtain that
\begin{align*}
4\langle y_1,z_1\rangle x_1-4\langle x_1,z_1 \rangle y_1
     & = (z_1\bar y_1 + y_1\bar z_1)x_1
+ (x_1\bar y_1)z_1 + (x_1\bar z_1)y_1 \\
     & \quad\; -(z_1\bar x_1+x_1\bar z_1)y_1
- (y_1\bar x_1)z_1-(y_1\bar z_1)x_1 \\
     & =(x_1{\bar y_1}-y_1{\bar x_1)}z_1
+(z_1{\bar y_1})x_1-(z_1{\bar x_1})y_1,
\end{align*}
and
\begin{align*}
4\langle y_2,z_2\rangle x_2-4\langle x_2,z_2 \rangle y_2
   & = x_2(\bar y_2z_2+\bar z_2y_2)+z_2(\bar y_2x_2) +y_2(\bar z_2x_2) \\
   & \quad\; -y_2(\bar x_2z_2+\bar z_2x_2) - z_2(\bar x_2y_2)-x_2(\bar z_2y_2) \\
   & = z_2(\bar y_2x_2-\bar x_2y_2)+x_2(\bar y_2z_2)-y_2(\bar x_2z_2),
\end{align*}
i.e.\ the expressions $S_{XY}Z-S_{YX}Z$ and $S^\prime_{XY}Z-S^\prime_{YX}Z$
coincide.

Brada and P\'ecaut-Tison's \cite{BraPec} expression for the
curvature tensor in terms of a ``cross product,'' coincides with
Brown and Gray's expression up to a factor $-c/4$
(cf.\ Remark in \cite[p.\ 145]{BraPec}, given without proof). To
prove that both expressions coincide it suffices to use the
property $(a,b,c)=-(\bar a,b,c)$ of the associator of $a,b,c\in
\mathbb{O}$, and four different expressions for $2\langle x,y\rangle a$
given above.

Then we have

\begin{proposition}
The curvature tensor of the Cayley planes is given by either
the expression \eqref{eq.4.1} or any of those obtained
using \eqref{eq.4.2}, \eqref{eq.4.3}, or \eqref{eq.4.4}.

Moreover, one has
\begin{equation}
\label{eq.4.5}
R_{XY}Z = \frac15 \sum_{0\leqslant j\leqslant 8} I_jR_{XY}I_jZ.
\end{equation}
\end{proposition}

\begin{proof}
Since the equivalence of the expressions~\eqref{eq.4.3} and~\eqref{eq.4.4}
was proved in \cite{Myk2} and the equivalence of
the curvature tensors defined by~\eqref{eq.4.2} and \eqref{eq.4.4}
was established above, only the equivalence of the curvature tensor
defined by either \eqref{eq.4.3} or
\eqref{eq.4.4} with the curvature tensor~\eqref{eq.4.1} remains to be proved.
We now prove this in two ways.

We know that the operator $R_{XY}$ is a linear combination of the
operators $I_{kl}$, $0\leqslant k<l\leqslant 8$, as the  isotropy
representation $\mathfrak{spin}(9) \to \mathrm{End}(T_pM)$ is the
$16$-dimensional spin representation of $\mathfrak{spin}(9)$. Since
for any fixed pair $kl$, $k,l=0,\dots,8$, $k\not=l$,
by~\eqref{eq.2.1} one has $\sum_{0\leqslant j\leqslant
8}I_jI_kI_lI_j = 5\,I_kI_l$, we get the formula \eqref{eq.4.5}.

On account of~\eqref{eq.4.3} and  \eqref{eq.4.5} we then obtain
that
\begin{align*}
5R&_{XY}Z  = \sum_{0\leqslant j\leqslant 8}I_jR_{XY}I_j Z \\
 & =-\frac{c}4\Bigl(3\sum_{0\leqslant j\leqslant 8}
 \bigl\{ g(I_jY,Z)I_j X - g(I_j X,Z)I_j Y \bigr\}
+ 9 \bigl\{ g(Y,Z)X - g(X,Z)Y \bigr\} \\
 & \qquad\quad\; + 2\sum_{0\leqslant i<j\leqslant 8}
\bigl\{ g(X,I_{ij}Z)I_{ij}Y - g(Y,I_{ij}Z)I_{ij} X \bigr\} \Bigr).
\end{align*}
Again using \eqref{eq.4.3} we then have
\[
R_{XY}Z=\frac{c}4\sum_{0\leqslant i<j\leqslant 8}
\bigl(g(Y,I_{ij}Z)I_{ij} X -
g(X,I_{ij}Z)I_{ij} Y \bigr),
\]
hence by virtue of Corollary~\ref{co.2.6} we deduce that
\[
R_{XY}Z =-\frac{c}4\sum_{0\leqslant i<j\leqslant 8} g(X,I_{ij}Y)I_{ij} Z,
\]
i.e.\ formula \eqref{eq.4.1}.

We can also prove the equivalence of the curvature
tensor~\eqref{eq.4.1} and that defined by~\eqref{eq.4.3}
considering for any vector fields $X,Y$ and the basis
of $2$-forms $\{\omega_{ij}, \sigma_{ijk}\}$ being
as in Section \ref{se.1},
Friedrich's expression \cite[Lemma 3.2]{Fri2}
\begin{equation}\label{eq.4.6}
8 \, X^\flat \wedge Y^\flat = \sum_{0\leqslant i<j\leqslant 8}
\omega_{ij}(X,Y)\omega_{ij}
+ \sum_{0\leqslant i<j<k\leqslant 8} \sigma_{ijk}(X,Y) \sigma_{ijk},
\end{equation}
where $X^\flat$ and
$Y^\flat$ denote the differential
$1$-forms metrically dual to $X$ and $Y$,
respectively. From \eqref{eq.4.6}, as a simple computation shows,
we obtain the formula
\begin{equation}\label{eq.4.7}
8\, \sum_{0\leqslant l\leqslant 8} (I_lX)^\flat \wedge (I_lY)^\flat
= 5 \sum_{0\leqslant i<j\leqslant 8 } \omega_{ij}(X,Y)\omega_{ij} - 3 \sum_{0\leqslant i<j<k\leqslant 8} \sigma_{ijk}(X,Y)\sigma_{ijk}.
\end{equation}
From equations \eqref{eq.4.6} and \eqref{eq.4.7} one easily
concludes.
\end{proof}

We omit for the sake of brevity the discussions
corresponding to the three next questions.

\begin{remark}
Another (longer but equivalent) expression in terms of
the operators $I_j$ for the curvature tensor
of the Cayley planes has been given in \cite[(4.18)]{Myk1}.
\end{remark}

\begin{remark}
Hangan gave in \cite[pp.\ 68--69]{Han}
another expression for the curvature tensor of
$\operatorname{\mathbb O \mathrm P}(2)$, this space viewed as a
differentiable manifold with three charts as in Besse \cite[p.\
91]{Bes}. The relation of his expression with those given above
remains as an open problem.
\end{remark}

\begin{remark}
The canonical metric on the open unit ball model of
$\operatorname{\mathbb O \mathrm P}(2)$, with
\[
B^2 = \{ (u,v) \in \mathbb{O}^2 : |u|^2 + |v|^2 < 1 \}
\]
has been recently found by Held, Stavrov and Van
Koten in \cite[Sect.\ 8]{HelStaKot}. It is given by
\[
g = \frac{c}4 \frac{|\mathrm{d} u|^2(1 - |v|^2) + |\mathrm{d} v|^2(1 - |u|^2)
+ 2\,\mathrm{Re}\bigl( u\bar v(\mathrm{d} v \,\mathrm{d} \bar u) \bigr)}
{(1 - |u|^2 - |v|^2{)}^2}, \quad c < 0.
\]
It would be interesting to relate their expression to the results of the
present paper. This also remains as an open problem.
\end{remark}

\section*{Appendix A}

We now give some comments on
Brada and P\'ecaut-Tison's expression of the canonical
$8$-form, showing that their $8$-form
$\omega$ (see~\cite[Def.\ 5.2]{BraPec0} or~\cite[Def.\
5.2]{BraPec}) is not $\mathrm{Spin}(9)$-invariant, and describing
some crucial gaps in their proof.

To define this form
$\omega$ they identify the space
$\mathbb{R}^{16}$ with the space
$\mathbb{O}^2$ and consider the cross product
$u\times v= \operatorname{Im}(\bar v u)=\frac12(\bar v u-\bar u v)$
of two elements
$u,v\in\mathbb{O}$ and the ``cross product'' of two vectors
$U,V\in\mathbb{O}^2$ as
\begin{equation}\label{eq.4.8}
U\times V=\bar u_1\times\bar v_1+ u_2\times v_2,
\quad\text{where}\quad
U=(u_1,u_2), V=(v_1,v_2).
\end{equation}
So the octonion
$U\times V=\operatorname{Im}(v_1\bar u_1)+\operatorname{Im}(\bar v_2 u_2)$
is pure imaginary. By Definition 5.2 in~\cite{BraPec0,BraPec} the
$8$-form
$\omega$ is given (up to a non-zero factor) by
\begin{equation}\label{eq.4.9}
\begin{split}
\omega(U_1,U_2,\dots,U_8)=2^{-7}
\sum_{\sigma\in S_8} \varepsilon(\sigma)
&[(U_{\sigma(1)}\times U_{\sigma(2)})(U_{\sigma(3)}\times
U_{\sigma(4)})]\\
\noalign{\vspace{-3mm}}
\cdot &[(U_{\sigma(5)}\times U_{\sigma(6)})(U_{\sigma(7)}
\times U_{\sigma(8)})].
\end{split}
\end{equation}
Putting
$$
a=U_1\times U_2,\quad  b=U_3\times U_4, \quad
c=U_5\times U_6,\quad  d=U_7\times U_8,
$$
as in~\cite{BraPec} we obtain that
\begin{align*}
2^4 & [(ab)(cd)+(ab)(dc)+(ba)(cd)+(ba)(dc) \\
    & \quad +(cd)(ab)+(cd)(ba)+(dc)(ab)+(dc)(ba)] \\
    & = 2^5 \operatorname{Re}[(ab)(cd)+(ab)(dc)+(ba)(cd)+(ba)(dc)] \\
    & = 2^5 \operatorname{Re}[(ab+ba)(cd+dc)] \\
    & = 2^7\operatorname{Re}(ab)\operatorname{Re}(cd),
\end{align*}
because all the elements $a,b,c,d$ are pure imaginary, so that, for
example, $\overline{(ab)(cd)}=(dc)(ba)$ and $\overline{ab}=ba$.
Remark also that in~\cite[p.\ 150]{BraPec} there is a misprint in
this formula (i.e.\ the last expression in~\cite{BraPec} is said to
be equal to $2^7 \operatorname{Re}[(ab)(cd)]$). Taking into account that by
definition the cross product in $\mathbb{O}$ and, consequently, the
``cross product''~\eqref{eq.4.8} in
$\mathbb{O}^2$ is skew-symmetric, we obtain that
\begin{equation}\label{eq.4.10}
\begin{split}
\omega(U_1,U_2,\dots,U_8)=
\sum_{\sigma\in S^*_8} \varepsilon(\sigma)
&\operatorname{Re}[(U_{\sigma(1)}\times U_{\sigma(2)})(U_{\sigma(3)}\times
U_{\sigma(4)})]\\
\noalign{\vspace{-3mm}}
\cdot &\operatorname{Re}[(U_{\sigma(5)}\times U_{\sigma(6)})(U_{\sigma(7)}
\times U_{\sigma(8)})].
\end{split}
\end{equation}
where $S^*_8=\{\sigma\in S_8: \sigma(2i-1)<\sigma(2i),
\sigma(1)<\sigma(3), \sigma(5)<\sigma(7), \sigma(1)<\sigma(5)\}$.
It is easy to verify that $\#(S^*_8)=8!/2^7=35\cdot 9$
and $\sigma(1)=1$
(its lowest number) for arbitrary $\sigma\in S^*_8$.

To prove that this form
$\omega$ is not $\mathrm{Spin}(9)$-invariant it is sufficient to show
that for the operator $I_{78}$ (which is an element of the Lie
algebra $\rho_*(\mathfrak{spin}(9))\subset\mathfrak{so}(16)$)
and some vectors $U_1,\dots U_8\in\mathbb{O}^2$ the following expression
\begin{equation}
\label{eq.4.11}
\omega(I_{78}U_1,U_2,\dotsc,U_8)
+\omega(U_1,I_{78}U_2,\dotsc,U_8)+\cdots
+\omega(U_1,U_2,\dotsc,I_{78}U_8)
\end{equation}
does not vanish.

Put $U_1=(0,u_0)$ and $U_2=(u_0,0),\dots,U_8=(u_6,0)$.
We will show that in this case the first term
$T_1$ in~\eqref{eq.4.11}
equals $63$ and that $|T_i|\leqslant 9$ for each other term
$T_i$, $i=2,\dots,8$. Since we have exactly
$7$ terms $T_i$ with $|T_i|\leqslant 9$,
the sum of all these eight terms is necessarily non-zero
if, for example, the eighth term $T_8\geqslant -8$.

Consider the first term $T_1=\omega(I_{78}U_1,U_2,\dots,U_8)$
in the sum~\eqref{eq.4.11}.
By~\eqref{eq.2.3} $I_{78}(0,u_0)=(u_7,0)$.
Since the product of any pair of elements of the basis
$B=\{u_0,\dots,u_7\}$ is an imaginary unit (up to a sign),
then each of the $35\cdot 9$ terms in
the expression~\eqref{eq.4.10} for
$\omega(I_{78}U_1,U_2,\dots,U_8)$ is given by
\begin{equation}
\label{eq.4.12}
\varepsilon(f)\varepsilon(\sigma')
\operatorname{Re}[(u_{\sigma'(1)}\bar u_{\sigma'(0)})
(u_{\sigma'(3)}\bar u_{\sigma'(2)})]
\operatorname{Re}[(u_{\sigma'(5)}\bar u_{\sigma'(4)})
(u_{\sigma'(7)}\bar u_{\sigma'(6)})],
\end{equation}
where $f$ is the unique bijection such that
$\sigma'=f\circ\sigma\circ f^{-1}$ and
$\sigma'$ is a permutation of the set $\{0,\dots,7\}$
with its natural ordering. Here
$\varepsilon(f)\varepsilon(\sigma')=\varepsilon(\sigma)$
and $\varepsilon(f)=-1$ because $f(1)=7$ and $f(i)=i-2$ for $i\geqslant 2$.
Since the product of all the elements of the basis $B$ is a real
number $\pm 1$ (see \eqref{eq.4.13} below),
then the term~\eqref{eq.4.12} is non-zero iff its first
factor of the form $\operatorname{Re}[\,\cdot\,]$
is non-zero. That is, we have $7$ possibilities for a choice
of the first pair $\{\sigma'(0),\sigma'(1)\}$ because
$\sigma'(0)=7$ ($\sigma(1)=1$) and $3$ possibilities for a choice
of the second pair $\{\sigma'(2),\sigma'(3)\}$ such
that $u_{\sigma'(2)}u_{\sigma'(3)}=\pm
u_{\sigma'(0)}u_{\sigma'(1)}$. Thus the number of non-zero
terms~\eqref{eq.4.12} equals $63$ because $\sigma(5)$
is the lowest number of the set $\{\sigma(5),\dots,\sigma(8)\}$
and then for a choice of $\sigma(6)$ one has $3$
possibilities. Remark that each such a term equals $\pm 1$
and that at least one of them is positive. This positive term
corresponds to the even permutation
$\sigma=(1, 2, 3, 8, 4, 5, 6, 7)$ with
$\sigma'=(7,0,1,6,2,3,4,5)$ because by~\eqref{eq.2.13}
\begin{align}
\operatorname{Re}[(u_0\bar u_7)(u_6\bar u_1)]
\operatorname{Re}[(u_3\bar u_2)(u_5\bar u_4)]
&=(-1)^4\operatorname{Re}[(1\cdot\mathbf{k}\mathbf{e})
(\mathbf{j}\mathbf{e}\cdot\mathbf{i})]
\operatorname{Re}[(\mathbf{k}\cdot\mathbf{j})
(\mathbf{i}\mathbf{e}\cdot\mathbf{e})] \notag \\
&=\operatorname{Re}[(\mathbf{k}\mathbf{e})
(\mathbf{k}\mathbf{e})]
\operatorname{Re}[(\mathbf{-i})(\mathbf{-i})]=1. \label{eq.4.13}
\end{align}

Now we will prove that all the non-zero terms~\eqref{eq.4.12}
coincide for any $\sigma'\in S_8$.
Taking into account the symmetries of the expression~\eqref{eq.4.12}
we can suppose that $\sigma'(0)=0$. Since all the elements
of the imaginary units set $B^0=B\setminus u_0$
anticommute and $\bar u=-u$ for such a unit, we can
rewrite the expression~\eqref{eq.4.12}
in the following form (up to a factor $\varepsilon(f)$)
$$
\phi(\sigma')=\varepsilon(\sigma')
\operatorname{Re}[(u_0u_{\sigma'(1)})
(u_{\sigma'(2)}u_{\sigma'(3)})]
\operatorname{Re}[(u_{\sigma'(4)}u_{\sigma'(5)})
(u_{\sigma'(6)}u_{\sigma'(7)})],
$$
where $\sigma'\in S_8$, $\sigma'(0)=0$. As we remarked above, this
expression is not zero iff its first factor $\operatorname{Re}[\,\cdot\,]$
is not zero. In this case
the algebra generated by the three imaginary units $u_{\sigma'(1)},
u_{\sigma'(2)},u_{\sigma'(3)}$ is isomorphic to the quaternion
algebra $\mathbb{H}$. In particular, the imaginary unit $u_{\sigma'(4)}$
is orthogonal to these three vectors and
$u_{\sigma'(3)}=\varepsilon_{12}u_{\sigma'(1)}u_{\sigma'(2)}$.
Therefore (\cite[Lect.\ 15, Lemma 1]{Pos})
there exists an automorphism $\Phi$ of $\mathbb{O}$
such that $\Phi(u_{\sigma'(1)})=u_1$, $\Phi(u_{\sigma'(2)})=u_2$,
and $\Phi(u_{\sigma'(4)})=u_4$. Then
$\Phi(u_{\sigma'(3)})=\varepsilon_{12} u_3$.
It is easy to see that $\Phi$ preserves
the set $B^0\cup(-B^0)$ and, consequently,
$\Phi(u_k)=\varepsilon_{u_k}^\Phi\sigma^\Phi(u_k)$, where
$\varepsilon_{u_k}^\Phi=\pm 1$ and $\sigma^\Phi$ is some
permutation in $S_8$ preserving $u_0$, and
$\prod_{k=0}^{7}\varepsilon_{u_k}^\Phi\cdot\varepsilon(\sigma^\Phi)=1$
(see the proof of Lemma~\ref{le.2.2}).
Thus
\begin{multline*}
\varepsilon(\sigma')\,
[(u_0u_{\sigma'(1)})
(u_{\sigma'(2)}u_{\sigma'(3)})]\cdot
[(u_{\sigma'(4)}u_{\sigma'(5)})
(u_{\sigma'(6)}u_{\sigma'(7)})] \\
= \varepsilon(\sigma'')
[(u_0u_1)(u_2u_3)]\cdot
[(u_4u_{\sigma''(5)})
(u_{\sigma''(6)}u_{\sigma''(7)})],
\end{multline*}
where $\sigma''=\sigma^\Phi\sigma'\in S_8$, because
$\varepsilon(\sigma'')=\varepsilon(\sigma')\varepsilon(\sigma^\Phi)=
\varepsilon(\sigma')\prod_{k=0}^{7}\varepsilon_{u_k}^\Phi$.
Note also that $\sigma''(j)=j$ if $j=0,1,2,3,4$ and
$\sigma''(j)\in\{5,6,7\}$ for $j=5,6,7$.
Since all the expressions in square brackets
are real and $\mathbf{i}\cdot\mathbf{j}\cdot\mathbf{k}=-1$, we have
$$
\phi(\sigma')=-\varepsilon(\sigma'')\cdot
[(u_4u_{\sigma''(5)})(u_{\sigma''(6)}u_{\sigma''(7)})].
$$
But $u_{\sigma''(4+i)}=u_{\tilde\sigma(i)}u_4$, $i=1,2,3$,
where $\tilde\sigma$ is some permutation in $S_3$. It is clear
that $\varepsilon(\sigma'')=\varepsilon(\tilde\sigma)$. Since
$(q_1\mathbf{e})(q_2\mathbf{e})=-\bar{q}_2q_1$
by~\eqref{eq.2.13}, we obtain that
$$
\phi(\sigma')=-\varepsilon(\tilde\sigma)
(-\bar u_{\tilde\sigma(1)})
(-\bar u_{\tilde\sigma(3)}u_{\tilde\sigma(2)})=
\varepsilon(\tilde\sigma)
u_{\tilde\sigma(1)}u_{\tilde\sigma(2)}u_{\tilde\sigma(3)}.
$$
Since the imaginary units $u_1,u_2,u_3$ anticommute, then the
non-zero value
$\phi(\sigma')$ $=\mathbf{i}\cdot\mathbf{j}\cdot\mathbf{k}=-1$
is independent of $\sigma'\in S_8$ and, consequently,
$\omega(I_{78}U_1,U_2,\dots,U_8)$ $=63$. Remark here that
the value $T_1=\omega(I_{78}U_1,U_2,\dots,U_8)$ is
calculated in~\cite[p.\ 150]{BraPec} but with a mistake.
By their calculations $T_1=35\cdot 9$ because the
calculations are based on the $\mathrm{SO}(8)\subset\mathrm{Spin}(9)$
invariance of the form $\omega$ given by~\eqref{eq.4.9}.
But as we will prove this form is not $\mathrm{Spin}(9)$-invariant.

Consider now the $i$th term
$T_i=\omega(U_1,\dots,I_{78}U_i,\dots)$, $2\leqslant i\leqslant 8$,
in~\eqref{eq.4.11}. By~\eqref{eq.2.3} for $0\leqslant k\leqslant
6$, one has $I_{78}(u_k,0)=(0,\pm u_{k'})$ with $1\leqslant
k'\leqslant 7$. Since the ``cross product'' $(x,0)\times(0,y)=0$
for any $x,y\in \mathbb{O}$ and $\sigma(1)=1$, $U_1=(0,u_0)$, then each
non-zero term in the expression~\eqref{eq.4.10} for
$\omega(U_1,\dots,I_{78}U_i,\dots)$ is determined by $\sigma\in
S^*_8$ such that $\sigma(2)=i$. This term is given by the following
expression
\begin{equation}
\label{eq.4.14}
\phi(\sigma)=\varepsilon(\sigma)
\operatorname{Re}[(\mp u_{(i-2)'}u_0)
(u_{\varphi(2)}\bar u_{\varphi(1)})]
\operatorname{Re}[(u_{\varphi(4)}\bar u_{\varphi(3)})
(u_{\varphi(6)}\bar u_{\varphi(5)})],
\end{equation}
where the six-point set $\{\varphi(1),\dots,\varphi(6)\}$
coincides with the set  $\{0,1,\dots,6\}\setminus\{i-2\}$.
If the term~\eqref{eq.4.14} is non-zero
then the first factor of the form
$\operatorname{Re}[\,\cdot\,]$ in~\eqref{eq.4.14} is non-zero.
That is, we have at most $3$ possibilities for a choice
of the second pair $\{\varphi(1),\varphi(2)\}$ because if
$\phi(\sigma)\not=0$ then
$u_{\varphi(2)}\bar u_{\varphi(1)}=\pm u_{(i-2)'}\subset B^0\cup(-B^0)$.
Thus the number of non-zero
terms~\eqref{eq.4.14} equals at most $9$ because $\sigma(5)$
is the lowest number of the set $\{\sigma(5),\dots,\sigma(8)\}$
and then for a choice of $\sigma(6)$ one has $3$
possibilities. Remark that each such a non-zero term equals
$\pm 1$.

Now to prove the non-invariance of the form $\omega$ it is
sufficient to find one positive term in the expression
for $T_8$. This positive term corresponds to the even
permutation $\sigma=(1, 8, 2, 3, 4, 5, 6, 7)$
because, by~\eqref{eq.2.13},
\begin{align*}
&\operatorname{Re}[(U_1\times I_{78}U_8)(U_2\times U_3)]\cdot\operatorname{Re}[(U_4\times U_5)(U_6\times U_7)] \\
& \qquad = \operatorname{Re}[((0,u_0)\times (0,u_6u_7))((u_0,0)\times(u_1,0))] \\
& \qquad \quad \cdot\operatorname{Re}[((u_2,0)\times(u_3,0))((u_4,0)\times(u_5,0))] \\
& \qquad = \operatorname{Re}[(u_7u_6\cdot u_0)(u_1\bar u_0)]\operatorname{Re}[(u_3\bar u_2)(u_5\bar u_4)] \\
& \qquad = \operatorname{Re}[(\mathbf{k}\mathbf{e}\cdot\mathbf{j}\mathbf{e})\mathbf{i}]
\operatorname{Re}[(\mathbf{k}\mathbf{j})(\mathbf{i}\mathbf{e}\cdot\mathbf{e})]
=1.
\end{align*}

Thus the $8$-form $\omega$ proposed in~\cite{BraPec0} and~\cite{BraPec}
is not $\mathrm{Spin}(9)$-invariant.

\begin{remark}
Using the method described above one can show that
only $T_2=-9$ and that all the other terms
$T_i=9$ for $i=3,\dots,8$.
Thus the expression~\eqref{eq.4.11} equals $108$.
\end{remark}

Note also that the proof of the invariance of the form
$\omega$ in~\cite{BraPec} contains some gaps.

First of all this proof is based on
the wrong proposition~\cite[Prop.\ 5]{BraPec}.
The proof of this proposition relies in turn on the fact
that the orthogonal transformations $T_a\colon \mathbb{O}\to\mathbb{O}$,
$x\mapsto axa$, of the space $\mathbb{O}$, where $a\in \operatorname{Im}\mathbb{O}$, $a^2=-1$, are pure
imaginary octonions of length $1$, generate a
group $G_T$ isomorphic to $\mathrm{SO}(8)$~(cf.\ \cite[p.151]{BraPec}).
But this is impossible because $T_a(u_0)=-u_0$ so that for any
$g\in G_T$ we have $g(u_0)=\pm u_0$.
Thus $G_T$ is locally isomorphic to $\mathrm{SO}(7)$ so that
$G_T\not\cong\mathrm{SO}(8)$.

Moreover, Prop.\ $5$ in~\cite{BraPec} asserts that the group
$G^*$ generated by certain one-parameter
subgroup and by the orthogonal transformations
$\tilde T_a\colon \mathbb{O}^2$
$\to\mathbb{O}^2$,
$(x_1,x_2)\mapsto(ax_1,x_2a)$, where
$a\in \operatorname{Im}\mathbb{O}$,
$a^2=-1$, are pure imaginary octonions of length
$1$, is isomorphic to the group
$\mathrm{Spin}(9)$. Now remark that by~\eqref{eq.4.10} their
$8$-form is
$\omega=\omega'\land\omega'$, i.e.\ it is the square of the
$4$-form $\omega'$ given by
$$
\omega'(U_1,U_2,U_3,U_4)=
\sum_{\sigma\in S_4} \varepsilon(\sigma) \operatorname{Re}[(U_{\sigma(1)}
\times U_{\sigma(2)})(U_{\sigma(3)}\times U_{\sigma(4)})].
$$
In~\cite[p.\ 152]{BraPec} it is proved that this
$4$-form $\omega'$ is $G^*$-invariant.
But we know (Brown and Gray \cite[Sect.\ 4.5]{BroGra}) that such a
non-zero $\mathrm{Spin}(9)$-invariant $4$-form cannot exist,
so that $G^*\not\cong\mathrm{Spin}(9)$.

\section*{Appendix B}

We now comment on Abe and Matsubara's expression of $\Omega^8$.
Remark first of all that using some computer calculations
we can obtain the expression for our $\mathrm{Spin}(9)$-invariant $8$-form
in some natural basis of $\mathbb{O}^2$. This expression contains
$702$ terms.

Abe and Matsubara attempted to
describe this $702$-terms expression for
$\Omega^8$ in their paper~\cite{AbeMat}
(see also the short announce by Abe~\cite{Abe}). The form $\Omega^8$
is exhibited there as a sum of eight $8$-forms
$\Omega^8_1,\dotsc,\Omega^8_8$. The combinatorial
descriptions of these forms given in~\cite{AbeMat}
are based on certain two $7\times8$
integer-valued matrices. But the combinatorial definitions of
these eight $8$-forms contain some mistakes,
for example the definition of the form
$\Omega^8_8$~(see \cite[p.\ 8]{AbeMat}) is not correct. Moreover,
the papers~\cite{Abe} and~\cite{AbeMat} contain different
expressions for the aforementioned form
$\Omega^8_8$. The expression given in~\cite{Abe} contains at most
$7\cdot7\cdot4=196$ terms (in some canonical basis)
though it is asserted in~\cite[p.12]{AbeMat} that
$\Omega^8_8$ contains
$336$ terms. Therefore we can not compare Abe-Matsubara's formula
and our formula for the canonical form
$\Omega^8$.

{\small

\providecommand{\bysame}{\leavevmode\hbox to3em{\hrulefill}\thinspace}

\noindent {\bf Authors's addresses}

\smallskip

\noindent M.C.L.: \!ICMAT\! (CSIC-UAM-UC3M-UCM), Departamento de
Geometr\'\i a y Topolog\'\i a, Facultad de Matem\'aticas,
Universidad Complutense de Madrid, Avenida Complutense s/n,
28040--Madrid, Spain. {\it E-mail\/}: {\tt mcastri@mat.ucm.es}

\smallskip

\noindent P.M.G.: Institute of Fundamental Physics,
CSIC, Serrano 113--bis, 28006--Madrid, Spain. {\it E-mail\/}: {\tt
pmgadea@iec.csic.es}

\smallskip

\noindent I.V.M.: Institute for Applied Problems of Mechanics
and Mathematics, National Academy of Sciences of Ukraine, Naukova
st., 3b, 79060--L'viv, Ukraine. {\it E-mail\/}: {\tt
mykytyuk\_i@yahoo.com}
}

\end{document}